\documentclass[12pt]{amsart}
\usepackage{mathtools}
\usepackage{amsfonts,amsmath,amssymb,amsthm}
\usepackage{latexsym}
\usepackage{mathrsfs}
\usepackage{tikz}
\usepackage{color}

\newtheorem{prop}{Proposition}[section]
\newtheorem{thm}{Theorem}[section]
\newtheorem{rem}{Remark}
\newtheorem{corollary}{Corollary}[section]
\newtheorem{definition}{Definition}[section]

\newtheorem{ex}{Example}[section]

\newcommand{\cH}{\mathcal{H}}

\newcommand{\cR}{\mathcal{R}}
\newcommand{\cX}{\mathcal{X}}

\newcommand{\alg}{\mathcal{A}^{gen}}

\begin{document}

\title[A Calabi--Yau algebra with $E_6$ symmetry and the 
 Clebsch--Gordan series of $sl(3)$]{A Calabi--Yau algebra with $E_6$ symmetry\\ and\\ the 
 Clebsch--Gordan series of $sl(3)$}

\author[N.Cramp\'e]{Nicolas Cramp\'e$^{\dagger}$}
\address{$^\dagger$ Institut Denis-Poisson CNRS/UMR 7013 - Universit\'e de Tours - Universit\'e d'Orl\'eans, 
Parc de Grandmont, 37200 Tours, France.}
\email{crampe1977@gmail.com }

\author[L.Poulain d'Andecy]{Lo\"ic Poulain d'Andecy$^{\ddagger}$}
\address{$^\ddagger$ Laboratoire de math\'ematiques de Reims UMR 9008, Universit\'e de Reims Champagne-Ardenne,
Moulin de la Housse BP 1039, 51100 Reims, France. }
\email{loic.poulain-dandecy@univ-reims.fr}

\author[L.Vinet]{Luc Vinet$^{*}$}
\address{$^*$ Centre de recherches math\'ematiques, Universit\'e de Montr\'eal,
P.O. Box 6128, Centre-ville Station,
Montr\'eal (Qu\'ebec), H3C 3J7, Canada.}
\email{vinet@CRM.UMontreal.ca}

\begin{abstract}
Building on classical invariant theory, it is observed that the polarised traces generate the centraliser $Z_L(sl(N))$ of the diagonal embedding of $U(sl(N))$ in $U(sl(N))^{\otimes L}$. 
The paper then focuses on $sl(3)$ and the case $L=2$. A Calabi--Yau algebra $\mathcal{A}$ with three generators is introduced and explicitly shown to possess a PBW basis and a certain central element. 
It is seen that $Z_2(sl(3))$ is isomorphic to a quotient of the algebra $\mathcal{A}$ by a single explicit relation fixing the value of the central element. 
Upon concentrating on three highest weight representations occurring in the Clebsch--Gordan series of $U(sl(3))$, a specialisation of $\mathcal{A}$ arises, involving the pairs of numbers 
characterising the three highest weights. 
In this realisation in $U(sl(3))\otimes U(sl(3))$, the coefficients in the defining relations and the value of the central element have degrees that correspond 
to the fundamental degrees of the Weyl group of type $E_6$. With the correct association between the six parameters of the representations and some roots of $E_6$, 
the symmetry under the full Weyl group of type $E_6$ is made manifest. 
The coefficients of the relations and the value of the central element in the realisation in $U(sl(3))\otimes U(sl(3))$ are expressed in terms of the fundamental invariant polynomials associated to $E_6$. 
It is also shown that the relations of the algebra $\mathcal{A}$ can be realised with Heun type operators in the Racah or Hahn algebra.
\end{abstract}

\maketitle

\vspace{3mm}




\setcounter{tocdepth}{1}
\tableofcontents

\section{Introduction}

In a nutshell, this paper introduces a Calabi--Yau algebra $\mathcal{A}^{gen}$ (according to the definition used in \cite{EG,G}) with striking features. 
The algebras which are going to play central roles are denoted:
\[\mathcal{A}^{gen}\ \rightsquigarrow\ \mathcal{A}\ \rightsquigarrow\ \mathcal{A}^{spec}\ .\]
The algebra $\mathcal{A}$ is the particular case, of special interest for our studies, of the generic Calabi--Yau algebra $\mathcal{A}^{gen}$ and the algebra $\mathcal{A}^{spec}$ is a specialisation of $\mathcal{A}$ corresponding to the choice of three highest weights of $sl(3)$.

Consider the centraliser of the diagonal embedding of $U(sl(3))$ in $U(sl(3))^{\otimes 2}$ which is non-Abelian and accounts for the degeneracies in the Clebsch--Gordan series. These degeneracies (\emph{i.e.} multiplicities) are the Littlewood--Richardson coefficients of $SL(3)$. We shall see that $\mathcal{A}$ surjects to this centraliser and shall observe moreover that the central parameters and the relation defining the kernel are mapped to elements whose degrees coincide with the fundamental degrees of the $E_6$ Weyl group. 
Furthermore, we shall see that this numerical observation is in fact the shadow of an actual $E_6$ symmetry occurring in the study of the direct sum decomposition of 
tensor products of two $sl(3)$-representations. Indeed, when the construction is restricted to triplets of highest weights representations of $sl(3)$ that are related in the Clebsch--Gordan series, 
the specialisation of $\mathcal{A}$ is found to be invariant under an action of the Weyl group of type $E_6$ on the parameters defining the highest weights. 
This action is described as the usual reflection representation once the parameters are correctly identified with some roots of $E_6$, and the specialised algebra can be entirely defined in terms of 
fundamental invariant polynomials of type $E_6$. Now before we get into the details of this, let us offer some background as a way of introduction.

A practical question in the decomposition of the tensor product of Lie or quantum algebra representations in irreducible components is the labelling of basis states; this is typically done by diagonalising a complete set of commuting operators that includes the total Casimir elements. Through the choice of a maximal Abelian subalgebra, the centraliser of the diagonal embedding in the tensor product of the universal algebra is the natural provider of labelling operators.

Interesting quadratic algebras have thus been identified by looking at the Racah problems, i.e. the case of three-fold products, for $U_q(sl(2))$ and $sl(2)$. 
Loosely speaking the centralisers of the diagonal embeddings in those cases are known as the Askey--Wilson \cite{GZh} and Racah \cite{GVZ} algebras (see also \cite{CPV}) which were originally 
introduced to encode the bispectral properties of the hypergeometric polynomials whose names they bear. 
The latter connection has led to the definition of Heun--Askey--Wilson \cite{BTVZ} and Heun--Racah \cite{BCTVZ} algebras which have their two generators realised by one of the bispectral operators and the other by the ``Heun'' operator given as the most general bilinear expression in the two bispectral operators \cite{GVZh}.

With an eye to generalisations, here we consider $sl(3)$ instead of the rank one algebras looked at so far. We examine the Clebsch--Gordan problem, that is we study the centraliser of the diagonal embedding of $U(sl(3))$ in two copies of this algebra. For $U_q(sl(2))$ and $sl(2)$, this problem is trivial; for $sl(3)$ however, it is not since the Clebsch--Gordan series exhibits degeneracies. This is one instance of 
``missing label problems'' where the centraliser must be called upon to complete the set of commuting operators. 
In the present case, the extra label will distinguish between identical modules occurring in the Clebsch--Gordan series. 
Our undertaking has led to the findings mentioned in the incipit. Of note is the identification of the cubic algebra $\mathcal{A}$ - a particular case of our master algebra $\mathcal{A}^{gen}$ - which maps surjectively to the centraliser of interest. Remarkably, this same algebra albeit with different central terms, has been identified by Lehrer and Racah \cite{racah} quite some time ago in the context of another well known ``missing label problem" involving $sl(3)$ or more precisely $su(3)$ in this case. This other ``missing label problem" (see for instance \cite{JMPW}) has to do with providing a complete labelling for the states that transform under an irreducible representation of $su(3)$ and 
are eigenvectors of the Casimir elements associated to the chain of subalgebras $su(3) \supset o(3) \supset o(2)$. 
A simple count shows that there is indeed one operator missing which has to be taken from the centraliser of $o(3)$ in $U(su(3))$. 
This is probably indicative of the larger relevance of $\mathcal{A}^{gen}$.

In general, the study of the diagonal centraliser of $U(\mathfrak{g})$, where $\mathfrak{g}$ is a semisimple or reductive Lie algebra, 
into an $L$-fold tensor product of $U(\mathfrak{g})$ can be put in perspective with the usual description of the centre of $U(\mathfrak{g})$. 
Indeed, the centre corresponds to the situation $L=1$, and it is described, through the Harish--Chandra isomorphism, as a (commutative) polynomial algebra with 
generators associated to the fundamental invariants of the Weyl group of $\mathfrak{g}$. Except if $\mathfrak{g}=sl(2)$ and $L=2$, the diagonal centraliser will not be commutative, 
due to the degeneracies appearing in the Clebsch--Gordan series. 

In the situation studied in detail in this paper ($\mathfrak{g}=sl(3)$ and $L=2$) but also in the situation $\mathfrak{g}=sl(2)$ and $L=3$, an interesting picture is starting to emerge.
In both cases, a Calabi--Yau algebra naturally appears instead of a commutative polynomial algebra (note that a polynomial algebra is in a sense the first example of a Calabi--Yau algebra). 
Another additional difficulty compared to the centre is that the commutative associated graded algebra of the diagonal centraliser is not generated by algebraically independent elements. 
This accounts for the need to take a quotient of the Calabi--Yau algebras in order to reach a complete description.
Again in both cases, the same phenomenon appears and leads to a rather natural quotient. 
Namely the quotient simply amounts to fix the value of the canonical central element of the Calabi--Yau algebra. 

We note that, as for the centre, the final description of the diagonal centraliser involves the fundamental invariants of a certain reflection groups, 
whose relevance to the problem was, 
at least for the authors, not easy to predict beforehand. For the situation considered here ($\mathfrak{g}=sl(3)$ and $L=2$), 
the type $E_6$ appears through a certain embedding of the root system of type $A_2\times A_2\times A_2$ into the root system of type $E_6$. 
Once this is done, it turns out quite surprisingly that the natural symmetry (under the Weyl group $W(A_2)\times W(A_2)\times W(A_2)$) actually extends to a symmetry under the full Weyl group of type $E_6$.
For the case $\mathfrak{g}=sl(2)$ and $L=3$, a similar symmetry also appears. Indeed, the Weyl group $W(D_4)$ appears through a certain embedding of the root system of type $A_1\times A_1\times A_1 \times A_1$ into 
the root system of type $D_4$ \cite{Zhedanov}. In both cases, the inclusion of root systems is natural in the affine extension of $E_6$ (respectively, $D_4$), which is a star-shaped affine Dynkin diagram. 

These star-shape affine Dynkin diagrams make their appearance in \cite{ELOR}, where is studied the quantum Hamiltonian reduction of tensor products of (certain quotients of) $U(sl_N)$ for various $N$. It is shown that these Hamiltonian reductions are isomorphic to spherical subalgebras of symplectic reflection algebras of certain types. The particular cases corresponding to the centralisers discussed in the previous paragraph are, respectively, $4$ copies of $U(sl_2)$ and $3$ copies of $U(sl_3)$, and for these cases, the corresponding symplectic reflection algebras are found to be of type, respectively, $D_4$ and $E_6$ (and of rank 1). Thus the algebras studied in the present paper can be seen as a presentation of the spherical subalgebra associated to the centraliser in \cite{ELOR} in the particular case of $U(sl_3)^{\otimes 3}$ making apparent the $E_6$ combinatorics.

The paper will unfold as follows. The Calabi--Yau algebra $\mathcal{A}^{gen}$ with three generators around which the article revolves is defined in Section 2. It is seen to be filtered with an appropriate degree assignment to the generators and shown explicitly to
have a PBW basis. The Calabi--Yau potential from which it derives is given together with the Casimir element. 
The polarised traces in $U(sl(N))^{\otimes{L}}$ are introduced in Section 3 where it is explained that their algebra denoted by $Z_L(sl(N))$ is the centraliser 
of the diagonal embedding of $U(sl(N))$ in its $L$-fold tensor product. 
In Section 4 we restrict to two copies of $sl(3)$ and use natural symmetry conditions to choose the generators of $Z_2(sl(3))$, 
the diagonal centraliser of $U(sl(3))$. Furthermore, a particular case $\mathcal{A}$ of $\mathcal{A}^{gen}$ is specified and shown to be realised in $U(sl(3))\otimes U(sl(3))$ 
via a surjective degree preserving morphism to $Z_2(sl(3))$. The kernel of this map is then identified so as to completely characterise the diagonal centraliser $Z_2(sl(3))$ as the 
quotient of $\mathcal{A}$ by an additional relation. It will have been observed that the central parameters of $\mathcal{A}$ and the additional relation have degrees that correspond to the 
fundamental degrees of the Weyl group of type $E_6$. The specialisation of the centraliser $Z_2(sl(3))$ that follows from picking three weights for $sl(3)$ labelled by the pairs of numbers $(m_1,m_2), (m'_1,m'_2), (m''_1,m''_2)$ is considered in Section 5. 
The effect is that the central parameters are replaced by expressions involving these numbers. An action of the Weyl group of $E_6$ on the parameters $(m_1,m_2), (m'_1,m'_2), (m''_1,m''_2)$ 
is obtained by associating these with some roots of $E_6$. It is then seen that the specialised polarised trace algebra $Z_L(sl(3))^{\it{spec}}$ is invariant under the action of the Weyl group of $E_6$ corresponding to its reflection representation. Indeed the expressions arising from the central terms are all given in terms of the fundamental invariant polynomials of $E_6$. Last in Section 6, it is shown that the relations of the algebra $\mathcal{A}^{gen}$ are satisfied by operators of the (generalised) Heun type in the Racah and Hahn algebras. The detailed relations between the parameters involved have been relegated to Appendix A. The paper ends with Section 7 that comprises some concluding remarks.

\section{A Calabi--Yau algebra with three elements.} \label{sec-A}

Let  $\mathcal{P}=\{a_0,a'_0,a_1,a_2,a_3,a_4,a_5,a_6,a_8,a_9\}$ be a set of indeterminates (which we also call central parameters)  and let $\mathbb{C}[\mathcal{P}]$ be the polynomial 
algebra in these indeterminates:
\[\mathbb{C}[\mathcal{P}]=\mathbb{C}[a_0,a'_0,a_1,a_2,a_3,a_4,a_5,a_6,a_8,a_9]\ .\]
\begin{definition}\label{def:cA}
The algebra $\alg$ is the algebra over $\mathbb{C}[\mathcal{P}]$ generated by the elements:
\[A,B,C\,,\]
with the following defining relations:
\begin{equation}\label{relAlg}
\begin{array}{l}
[A,B]=C\,,\\[0.5em]
[A,C]= a_0 B^2 +a_1\{A,B\}+ a_2 A^2 + a_4 B + a_5 A+a_8\,,\\[0.5em]
[B,C]= a'_0 A^3 -a_1 B^2-a_2\{A,B\}+a_3A^2 -a_5 B+a_6A+a_9\,,
\end{array}
\end{equation}
where $\{A,B\}=AB+BA$ stands for the anticommutator.
\end{definition}
Recall that an algebra $F$ is filtered if it has an increasing sequence of subspaces:
\[F_0\subset F_1\subset F_2\subset\dots\subset F_n\subset\dots\quad\ \ \ \ \text{such that}\ \ F=\bigcup_{n\geq 0}F_n\,,\]
compatible with the multiplication in the sense that $F_n\cdot F_m\subset F_{n+m}$\ . 
Defining $F_{-1}=\{0\}$, the associated graded algebra is:
\[gr(F)=\bigoplus_{n\geq 0} F_n/F_{n-1}\,,\]
with well-defined multiplication $(x+F_{n-1})(y+F_{m-1})=xy+F_{n+m-1}$, for $x\in F_n$ and $y\in F_m$. 
The Hilbert--Poincar\'e series of $F$ records the dimensions of the graded components $F_n/F_{n-1}$.

The algebra $\alg$ is filtered when we define the degree of each generator in the following way:
\[deg(A)=3\,,\ \ deg(B)=4\,,\ \ deg(C)=6\ .\]
The subspace $\alg_n$ consists of elements of degree less or equal to $n$. The chosen values of the degrees of $A,B,C$ are natural in view of the 
later realisation of $\alg$ in $U(sl(3))\otimes U(sl(3))$. Implicitly here, one can see the elements in the ground ring $\mathbb{C}[\mathcal{P}]$ as being of degree 0.
When we shall specialise the algebra in the following, the coefficients will acquire the degree that explains their labelling.   

From the defining relations, one sees immediately that the graded algebra $gr(\alg)$ is commutative. As a consequence of the next proposition, $gr(\alg)$ is in fact isomorphic to the polynomial algebra over $\mathbb{C}[\mathcal{P}]$ on three commuting variables $A,B,C$.

\begin{prop}\label{prop-basisA'}
The algebra $\alg$ has a ``PBW basis'', that is, the algebra $\alg$ is free over $\mathbb{C}[\mathcal{P}]$ with basis:
\begin{equation}\label{basisA}
\{A^{\alpha}B^{\beta}C^{\gamma}\ \}_{\alpha,\beta,\gamma\in\mathbb{Z}_{\geq 0}}\ .
\end{equation}
Its Hilbert--Poincar\'e series as a filtered $\mathbb{C}[\mathcal{P}]$-algebra is:
\[\frac{1}{(1-t^3)(1-t^4)(1-t^6)}\ .\]
\end{prop}
\begin{proof}
The Hilbert--Poincar\'e series follows from the knowledge of the basis. To prove that (\ref{basisA}) forms a basis, we use the standard diamond lemma approach from \cite{Be}. 
First we define a partial ordering on the set of words in the generators of $\alg$ and order the three generators by
\[A<B<C\,.\]
Then for a given word $X_1\dots X_k$ in the generators, we already have its degree and we define also its misordering index as
\[n(X_1\dots X_k)=Card\{(i,j)\ |\ 1\leq i<j\leq k\ \text{and}\ X_j<X_i\}\,. \]
Finally for two words $X$ and $Y$ in the generators, we state that $X<Y$ if either $deg(X)<deg(Y)$ or $Y$ is obtained from $X$ by permuting 
the letters and satisfies $n(X)<n(Y)$. 

This defines a partial order on the set of words in the generators which satisfies the two natural conditions required in \cite{Be}. 
Namely, that there is only a finite number of words smaller than a given word $X$ and that for two given words such that $X<Y$, we have $ZXZ'<ZYZ'$ for any two words $Z,Z'$.

Then we interpret the defining relations (\ref{relAlg}) as instructions for rewriting a given word as a linear combination or ordered words in the generators. Explicitly, the instructions are:
\[\begin{array}{c} BA=AB-C\,,\quad CA=AC-(a_0 B^2+a_1(2AB-C)+ a_2 A^2 + a_4 B + a_5 A+a_8)\,,\\[0.5em]
\ CB=BC-\bigl(a'_0 A^3 -a_1 B^2-a_2(2AB-C)+a_3A^2-a_5 B+a_6A+a_9\bigr)\ .\end{array}\]
This set of instructions is compatible with the partial order $<$, in the sense that when applying one of these instructions to a subword of a word, we obtain a linear combination of words which are strictly smaller
in the partial ordering.

The diamond lemma formulated in \cite{Be} asserts that the set of ordered words in $(\ref{basisA})$ is a basis if all ambiguities are resolvable. In our situation, there is only one non-trivial ambiguity to check and the verification to make is the following. Starting from the word $CBA$, we have different possibilities of rewriting it (in fact two, depending on whether we start with reordering $CB$ or $BA$), and we must check that they result in the same linear combination. This is a straightforward verification for which we omit the details.
\end{proof}

\begin{rem}\label{rem1}
 The algebra $\alg$ is linked to algebras studied previously: the specialisation $a_0=a_0'=a_3=0$ corresponds to the Racah algebra \cite{Zhedanov} 
 and the specialisation $a_0=0$ is a particular case of the Heun--Racah algebra \cite{BCTVZ}. 
 The specialisation $a_0=0$ has been also studied in the context of superintegrable models \cite{Mar}.
 There exists also a realisation of the algebra $\alg$ in terms
 of the generators of the Racah algebras: relations \eqref{relAlg} of $\alg$ are the ones between two different Heun--Racah operators (see Section \ref{sec:racah}).
 There is another realisation in terms of the Hahn generators. In this case, the relations of the algebra $\alg$ encode the relations between a Heun--Hahn operator \cite{VZ} and 
 a more generic operator (see Section \ref{sec:hahn}).
\end{rem}

\subsection{Parameters and specialisations.} 

Equivalently, the algebra $\alg$ can be defined directly over $\mathbb{C}$, by promoting the parameters in $\mathcal{P}$ to generators with the additional defining relations that they are central. 
As a consequence of Proposition \ref{prop-basisA'}, a basis of $\alg$ over $\mathbb{C}$ consists of ordered monomials in elements in $\mathcal{P}$ and $A,B,C$, such as:
\[\{\ p_1^{i_1}\dots p_{10}^{i_{10}}A^{\alpha}B^{\beta}C^{\gamma}\ \}_{i_1,\dots,i_{10},\alpha,\beta,\gamma\in\mathbb{Z}_{\geq 0}}\,,\]
for any renaming of elements of $\mathcal{P}=\{a_0,a'_0,a_1,a_2,a_3,a_4,a_5,a_6,a_8,a_9\}$ as $\{p_1,\dots,p_{10}\}$. In this point of view, one can assign also a degree to the central parameters in $\mathcal{P}$, 
such that the algebra $\alg$ over $\mathbb{C}$ remains filtered with a commutative polynomial algebra as its graded algebra. 
It is immediate to see that the degrees of the parameters $a'_0$ and $a_i$ must then be such that the r.h.s. of the defining relations are of degree equal at most to the degree of the corresponding l.h.s. minus 1. 
One finds that the maximal values for the degrees of the parameters $a_i$ must be $i$. The degree of $a'_0$ must be 0. It justifies \textit{a posteriori} the names of the parameters.
These are the degrees which will appear naturally in the realisation in $U(sl(3))\otimes U(sl(3))$.

\vskip .2cm
Now, let $R$ be any complex algebra and consider any map $\vartheta$ from $\mathcal{P}$ to $R$. This is uniquely extended to a morphism of algebra:
\[\vartheta\ :\ \mathbb{C}[\mathcal{P}]\to R\ .\]
Thus, corresponding to $\vartheta$, there is a specialisation of $\alg$ which is $\mathcal{A}_{\vartheta}=\alg\otimes _{ \mathbb{C}[\mathcal{P}]} R$.  The algebra $\mathcal{A}_{\vartheta}$ is an algebra over $R$ and it is simply obtained from $\alg$ by replacing any central parameter $p$ in $\mathcal{P}$ by its chosen value $\vartheta(p)$ in $R$. From Proposition \ref{prop-basisA'}, we have at once that $\mathcal{A}_{\vartheta}$ is free over $R$ with basis consisting of ordered monomials in $A,B,C$:
\[\{A^{\alpha}B^{\beta}C^{\gamma}\ \}_{\alpha,\beta,\gamma\in\mathbb{Z}_{\geq 0}}\ .\]
In particular, we are going to consider the situation where $R$ is another polynomial ring (possibly with less indeterminates), and also the situation where $R=\mathbb{C}$ (for which the central parameters in $\mathcal{P}$ are simply replaced by complex numbers).

\subsection{Algebras with a Calabi--Yau potential}\label{potential}

We wish to point out that the algebra $\alg$ can be derived from a potential and that it shares in this respect a property of most Calabi-Yau algebras \cite{G}, \cite{BRS} of dimension 3. Let $F=\mathbb{C}\langle x_1, x_2, \dots x_n\rangle$ be a free associative algebra with $n$ generators and $F_{cycl}$ the quotient vector space $F/[F,F]$. $F_{cycl}$ has the cyclic words $[x_{i_1} x_{i_2}\dots x_{i_r}]$ as basis. The map $\frac{\partial}{\partial x_j}: F_{cycl} \rightarrow F$ is such that
\begin{equation}
\frac{\partial [x_{i_1}, x_{i_2},\dots x_{i_r}] }{\partial x_j}= \sum_{\{s|i_s=j\}} x_{i_s +1}x_{i_s +2}\dots x_{i_r}x_{i_1}x_{i_2}\dots x_{i_s -1} 
\end{equation} 
and is extended to  $F_{cycl}$ by linearity on combinations of cyclic words. Let $\Phi(x_1,\dots x_n) \in F_{cycl}$. An algebra whose defining relations are given by 
\begin{equation}
\frac{\partial\Phi}{\partial x_j}=0, \qquad j=1,\dots n
\end{equation}
is said to derive from the potential $\Phi$. Now let $x_1=A$, $x_2=B$ and $x_3=C$ and take
\begin{equation}\label{eq:Phi}
\Phi=[ABC]-[BAC]-\frac{a'_0}{4}[A^4]-\frac{a_3}{3}[A^3]+\frac{a_0}{3}[B^3]+a_2[A^2B]+a_1[AB^2]
\end{equation}
 \[-\frac{a_6}{2}[A^2]+a_5[AB]+\frac{a_4}{2}[B^2]-\frac{1}{2}[C^2]-a_9[A]+a_8[B].\]
With this, it is not hard to see that the relations \eqref{relAlg} of $\alg$ are given by 
$\frac{\partial \Phi}{\partial C}=0$, $\frac{\partial \Phi}{\partial B}=0$, $\frac{\partial \Phi}{\partial A}=0$.

\begin{rem}
 As mentioned in Remark \ref{rem1}, $\alg$ is a generalisation of the Racah algebra. Therefore, from the result stated above, the relations of the Racah algebra can be derived from
 the Calabi--Yau potential \eqref{eq:Phi} with $a_0=a_0'=a_3=0$. 
\end{rem}

\vskip .2cm
\paragraph{\textbf{PBW basis and central element.}} We still consider an algebra derived from a potential with three generators $x_1=A$, $x_2=B$ and $x_3=C$, and we give a degree to these generators:
\[deg(A)=a\,,\ \ deg(B)=b\,,\ \ deg(C)=c\,,\ \ \ \ \ \text{and}\ \ \ d:=a+b+c\ .\]
We take a potential of the form:
\[\Phi=\Phi^{(d)}+\Phi^{<d}\,,\ \ \ \ \text{where}\ \Phi^{(d)}=[ABC]-[BAC]\,,\]
and $\Phi^{<d}$ is any linear combination of cyclic words such that $deg(\Phi^{<d})<d$. Then the defining relations of the algebra derived from $\Phi$ are:
\[[A,B]=-\frac{\partial\Phi^{<d}}{\partial C}\,,\ \ \ \ [A,C]=\frac{\partial\Phi^{<d}}{\partial B}\,,\ \ \ \ [B,C]=-\frac{\partial\Phi^{<d}}{\partial A}\,.\]
In the more general setting where $\Phi^{(d)}$ is a so-called homogeneous Calabi--Yau potential of degree $d$ (of which $[ABC]-[BAC]$ is a particular case), it is proven in \cite[Theorem 3.3.2(i)]{EG} (see also \cite{BT}) that the Hilbert--Poincar\'e series of the resulting algebra is:
\[\frac{1}{(1-t^a)(1-t^b)(1-t^c)}\ .\]
For $\Phi^{(d)}=[ABC]-[BAC]$ (and arbitrary $\Phi^{<d}$), a simple proof of a PBW basis and thus of the form of the Hilbert--Poincar\'e series follows exactly the same lines as the proof of Proposition \ref{prop-basisA'}. Namely, the fact that there is no ambiguity when rewriting $CBA$ in the ordered basis becomes equivalent to:
\[A\frac{\partial\Phi^{<d}}{\partial A}+B\frac{\partial\Phi^{<d}}{\partial B}+C\frac{\partial\Phi^{<d}}{\partial C}=\frac{\partial\Phi^{<d}}{\partial A}A+\frac{\partial\Phi^{<d}}{\partial B}B+\frac{\partial\Phi^{<d}}{\partial C}C\ .\]
This equality is easy to check if $\Phi^d$ is a single cyclic word and thus extends to a general combination of cyclic words by linearity.

Again in the general setting  where $\Phi^{(d)}$ is a homogeneous Calabi--Yau potential of degree $d$, it is proven in \cite[Theorem 3.3.2(ii)]{EG} 
that there is a non-trivial central element $\Omega$ of degree $\leq d$. It is pointed out in \cite{EG} that there is no explicit expression for $\Omega$ in the general case. 
However, in a given particular case, a straightforward inspection allows to find this central element. 

\begin{rem}
It is suggested in \cite{EG} that the quotient of the algebra derived from a Calabi--Yau potential by the ideal generated by the central element $\Omega$ can 
be viewed as a non-commutative analogue of a Poisson algebra. It is remarkable that the realisation that we have later in $U(sl(3))\otimes U(sl(3))$ of the 
algebra $\alg$ factors precisely through such a quotient.
\end{rem}

\subsection{The Casimir element of $\alg$.}
As a consequence of the PBW basis, working with the algebra $\alg$ is an algorithmically straightforward procedure. We can use the defining relations as rewriting instructions for ordering any word in the PBW basis. Using this straightforward approach, we check that the following element is central in $\alg$: 
\begin{equation}\label{eq:CAP}
\Omega =x_1A+x_2B+x_3A^2+x_4\{A,B\}
+x_5 B^2
      +x_6A^3+x_7 ABA+x_8 BC-x_{8}AB^2
      +x_{9} A^4+x_{10} B^3+C^2
\end{equation}
where
\[\begin{array}{l}
x_1= \frac{1}{6}( 12 a_9 +3 a_0 a_5 a'_0 + 2 a_4 a_3- a_4 a'_0 a_1 - 
    6 a_6 a_1)\,,\\[0.4em]
x_2= \frac{1}{3}(- 6 a_8-3 a_4 a_2 + a_0 a_4 a'_0 + a_0 a_6  + 6 a_5 a_1)\,,\\[0.4em] 
x_3= \frac{1}{6}( 6 a_6 +3 a_4 a'_0 + 3 a_0 a_2 a'_0 - 4 a_3 a_1 - a'_0 a_1^2)\,,\\[0.4em]
x_4=\frac{1}{3}(-3a_5+a_0a_3+3a_2a_1+a_0a'_0a_1)\,,\\[0.4em]
x_5= \frac{1}{3}(-3 a_4 - 4 a_0 a_2 + a_0^2 a'_0 + 6 a_1^2)\,,\\[0.4em]
x_6= \frac{1}{3}(2 a_3 - a'_0 a_1)\,,\ \ \ x_7= -2 a_2 + a_0 a'_0\,,\\[0.4em]
x_8 =  2 a_1\,,\ \ \  x_{9} =\frac{1}{2}a'_0\,,\ \ \  x_{10}= -\frac{2}{3}a_0\ .
  \end{array}\]
We call the element $\Omega$ the Casimir element of $\alg$. By direct investigations, we prove that it is the (non-scalar) central element of $\alg$ of minimal degree (its degree is $12$).

This central element can be rewritten to underscore the link with the Calabi--Yau potential \eqref{eq:Phi}. Indeed, by replacing the cyclic words in the Calabi--Yau potential $\Phi^{<d}$ by the 
symetrisations of this word, $-2\Phi^{<d}$ leads to the Casimir element up to terms of lower degrees.


\section{Polarised traces and the diagonal centraliser of $U(sl(N))$  \label{sec:pt}}

\subsection{The Lie algebra $sl(N)$\label{sec:lie}} The Lie algebra $gl(N)$ is generated by the elements $\{\, e_{pq}\, |\, 1\leq p,q\leq N\, \}$ which satisfy
the following defining relations
\begin{equation}\label{eq:defglN}
 [e_{pq}, e_{rs} ] = \delta_{qr} e_{ps} - \delta_{sp} e_{rq}\qquad \text{for } 1\leq p,q,r,s\leq N.
\end{equation}
The Lie algebra $sl(N)$ is the quotient of $gl(N)$ by the supplementary relation:
\[\sum_{p=1}^ne_{pp}=0\ .\]

\subsection{Polarised traces} We consider the tensor product of $L$ copies of $U(gl(N))$. For $a\in\{1,\dots,L\}$, we introduce the notation:
\begin{equation}\label{eijeps}
e_{pq}^{(a)}=1^{\otimes a-1} \otimes e_{pq} \otimes 1^{\otimes L-a}\in U(gl(N))^{\otimes L}.
\end{equation}
The following map
\begin{eqnarray}\label{eq:delta}
\delta\quad : \quad  U(gl(N))&\rightarrow& U(gl(N))^{\otimes L}\\
 e_{pq} &\mapsto&  \sum_{a=1}^L  e_{pq}^{(a)}\nonumber
\end{eqnarray}
extends to an algebra homomorphism. The map $\delta$ is called the diagonal map and its image is the diagonal embedding of $U(gl(N))$ in $U(gl(N))^{\otimes L}$.

Let us define the following elements of $U(gl(N))^{\otimes L}$ (from now on the summation over repeated indices will be implicit):
\[T^{(a_1,\dots,a_d)}=e_{i_2i_1}^{(a_1)}e_{i_3i_2}^{(a_2)}\dots e_{i_1i_d}^{(a_d)}\,,\ \ \ \ \ \ a_1,\dots,a_d\in\{1,\dots,L\}\ .\]
By a direct computation, we can show that these elements commute with $\delta(e_{pq})$. Therefore they are in the centraliser of the diagonal embedding of $U(gl(N))$ in $U(gl(N))^{\otimes L}$. 
The elements $T^{(a_1,\dots,a_d)}$ are called polarised traces. 

\begin{definition}
We denote $Z_L(gl(N))$ (respectively, $Z_L(sl(N))$) the subalgebra of $U(gl(N))^{\otimes L}$ (respectively, of $U(sl(N))^{\otimes L}$) generated by all polarised traces $T^{(a_1,\dots,a_d)}$ with $d\geq 1$ and $a_1,\dots,a_d\in\{1,\dots,L\}$.
\end{definition}

\subsection{Results from classical invariant theory} For the classical results on invariant theory that we recall below, we refer to \cite{BS, Pr, Ra} (see also \cite{Dr,Fo} and references therein). 

The first fundamental theorem on the algebra of polynomial functions on $gl(N)\times\dots\times gl(N)$ ($L$ times) which are invariant under simultaneous conjugation by $GL(N)$ asserts that it is generated by the functions:
\begin{equation}\label{functionPT}
(M_1,\dots,M_L)\mapsto \text{Tr}(M_{a_1}M_{a_2}\dots M_{a_d})\,,\ \ \ \text{with $d\geq 1$ and $a_1,\dots,a_d\in\{1,\dots,L\}$.}
\end{equation}
In general, it is furthermore known that a set of generators is obtained by restricting to these functions with $d\leq N^2$. And it is conjectured that it is enough to take $d\leq \displaystyle\frac{1}{2}N(N+1)$. This conjecture is known to be true if $N\in\{2,3,4\}$.

With a standard reasoning that we sketch now, these results provide some information on the algebra of polarised traces in the deformed setting of universal enveloping algebras.

Take $U=U(gl(N))^{\otimes L}$. It is a filtered algebra (the degree of each generator $e_{pq}^{(a)}$ is 1) and recall that its associated graded algebra is the commutative algebra of polynomials in $e_{pq}^{(a)}$. So we can naturally identify the algebra of polynomial functions on $gl(N)\times\dots\times gl(N)$ with this graded commutative algebra:
\[gr(U)=\bigoplus_{n\geq 0} \mathbb{C}_n[e_{pq}^{(a)}]\ ,\]
where $\mathbb{C}_n[e_{pq}^{(a)}]$ is the space of homogeneous polynomials of degree $n$. The identification sends $e_{pq}^{(a)}$ to the linear form giving the $(q,p)$ coordinate of the $a$-th matrix.

The Lie algebra $gl(N)$ acts diagonally on $U=U(gl(N))^{\otimes L}$ (namely, by composing the diagonal map $\delta$ followed by the commutator) and this respects the filtration, $gl(N)\cdot U_n\subset U_n$, so it induces an action on $gr(U)$. Under the above identification, a direct verification shows that this action of $gl(N)$ on $gr(U)$ corresponds to the derivative of the conjugation action of $GL(N)$ on polynomial functions. 
So in particular, invariant functions correspond to elements of $gr(U)$ commuting with all elements $\delta(e_{pq})$. Moreover, the generating invariant functions (\ref{functionPT}) correspond to the images in $gr(U)$ of the polarised traces denoted $T^{(a_1,\dots,a_d)}$.

We call the diagonal centraliser, for brevity, the centraliser of the diagonal embedding of $U(gl(N))$ in $U(gl(N))^{\otimes L}$. Using the result on classical invariants, what we just explained is that the image of the diagonal centraliser in $gr(U)$ is generated by the images of the elements $T^{(a_1,\dots,a_d)}$.  
So in other words, for any element $z$ of degree $n$ in the diagonal centraliser, we have that modulo $U_{n-1}$ the element $z$ can be obtained as a certain polynomial in $T^{(a_1,\dots,a_d)}$. By an easy induction, this gives that any element in the diagonal centraliser is generated by the polarised traces $T^{(a_1,\dots,a_d)}$.

Besides we have also noted in the preceding subsection that the polarised traces $T^{(a_1,\dots,a_d)}$ commute with the diagonal embedding of $U(gl(N))$. So we can conclude that we have the following corollary of the classical invariant theory. 
\begin{corollary}
The algebra $Z_L(gl(N))$ of polarised traces coincides with the centraliser of the diagonal embedding of $U(gl(N))$ in $U(gl(N))^{\otimes L}$.
\end{corollary}

Further, with the same arguments, we can also extract a finite set of generators for $Z_L(gl(N))$, namely the set of polarised traces $T^{(a_1,\dots,a_d)}$ with $d\leq N^2$. Moreover, for $N\in\{2,3,4\}$, it is enough to take $d\leq \displaystyle\frac{1}{2}N(N+1)$. 

\begin{rem}
All this subsection can be formulated in an obvious way for $Z_L(sl(N))$ with similar statements (one just has to remove the polarised traces of degree $d=1$ from the set of generators).
\end{rem}

\section{The diagonal centraliser in two copies of $sl(3)$}

From now on, we will mainly focus on the situation $N=3$ and $L=2$ of the preceding subsection.

\subsection{Generators of $Z_2(sl(3))$.} From the results in classical invariant theory, we have that $Z_2(sl(3))$ is generated by the polarised traces of degree less or equal to $6$. More precisely, from \cite{Dr} (see also references therein), one can extract the following information. As a subalgebra of $U(sl(3))\otimes U(sl(3))$, the algebra $Z_2(sl(3))$ is filtered. The Hilbert--Poincar\'e series of $Z_2(sl(3))$ is:
\begin{equation}\label{HP-sl3}
\frac{1-t_1^6t_2^6}{(1-t_1^2)(1-t_1t_2)(1-t_2^2)(1-t_1^3)(1-t_1^2t_2)(1-t_1t_2^2)(1-t_2^3)(1-t_1^2t_2^2)(1-t_1^3t_2^3)}\ .
\end{equation}
This is the Hilbert--Poincar\'e series of a subalgebra of the polynomial algebra in the generators $e_{pq}^{(a)}$. 
The coefficient in front of $t_1^kt_2^l$ is the dimension of the graded part of $Z_2(sl(3))$ in degree $k$ in the generators $e_{pq}^{(1)}$ and $l$ in the generators $e_{pq}^{(2)}$. 

A set of generators of $Z_2(sl(3))$ is
\begin{equation}\label{gen-sl3}
T^{(1,1)},\ T^{(1,2)},\ T^{(2,2)},\ T^{(1,1,1)},\ T^{(1,1,2)},\ T^{(1,2,2)},\ T^{(2,2,2)},\ T^{(1,1,2,2)} \quad \text{and} \quad T^{(1,1,2,2,1,2)}\ .
\end{equation}
Moreover, in the graded commutative algebra, we have that the first eight generators are algebraically independent, and thus, looking at the numerator of the Hilbert--Poincar\'e series, 
we see that there is an algebraic relation among the generators in degree $12$ (more precisely, in degree $(6,6)$). We shall describe this relation in Subsection \ref{subsec-pres}.

In classical invariant theory, there is some freedom in the choice of the last generator. However, not every polarised trace of degree $(3,3)$ can serve as generator. 
For example the choice $T^{(1,1,1,2,2,2)}$ is not possible (the rule is that at least two consecutive indices must be equal but not three). In our setting this ambiguity disappears as we can simply remove the last generator from the list since it turns out that $T^{(1,1,2,2,1,2)}$ can be expressed in terms of the other generators. Indeed, one can check that:
\[2T^{(1,1,2,2,1,2)}+[T^{(1,1,2)},T^{(1,1,2,2)}]+\frac{1}{3}T^{(1,1,1)}T^{(2,2,2)}-T^{(1,1,2)}T^{(1,2,2)}-T^{(1,1,2,2)}T^{(1,2)}\ \]
has no term of degree 6, and thus can be expressed in terms of the first eight generators (the explicit expression is $T^{(1,1)}T^{(1,2,2)}-T^{(2,2)}T^{(1,1,2)}-6T^{(1,1,2,2)}+2T^{(1,1)}T^{(2,2)}-12(T^{(1,1,2)}+T^{(1,2,2)})-16 T^{(1,2)}$).

\begin{rem}
It is interesting to note that in our non-commutative setting, $Z_2(sl(3))$ is generated by polarised traces of degree up to $4$, since the last generator is obtained by taking a commutator of two other generators (which is a new phenomenon compared to the classical setting).
\end{rem}

\subsection{Automorphisms.}\label{subsec-aut} Consider the following maps on the generators of $U(sl(3))\otimes U(sl(3))$:
\[\tau\ :\ e_{ij}^{(1)}\mapsto -e_{ji}^{(1)}\ \ \ \ \ \text{and}\ \ \ \ \ \ e_{ij}^{(2)}\mapsto -e_{ji}^{(2)}\,,\]
\[(1,2)\ :\ \ \ e_{ij}^{(1)}\mapsto e_{ij}^{(2)}\ \ \ \ \ \text{and}\ \ \ \ \ \ e_{ij}^{(2)}\mapsto e_{ij}^{(1)}\,.\]
They clearly extend to automorphisms of $U(sl(3))\otimes U(sl(3))$ and, since they leave stable the diagonal embedding of $U(sl(3))$, they restrict to automorphisms of $Z_2(sl(3))$ (one can also check directly that they preserve the algebra generated by the polarised traces).

The third automorphism of $Z_2(sl(3))$ we need is a bit more involved to define. First, we define the map $\theta$ from $U(sl(3))$ to $U(sl(3))\otimes U(sl(3))$ which is the composition of the antipode and of the diagonal embedding. It is an antiautomorphism defined on $g\in sl(3)$ by $\theta(g)= -g\otimes 1-1\otimes g$. Then we define a linear map by
\[\begin{array}{lrcl}(2,3)\ :\ \ \ & U(sl(3))\otimes U(sl(3)) & \to & U(sl(3))\otimes U(sl(3))\\[0.4em]
 & a\otimes b & \mapsto & (a\otimes 1).\theta(b)\end{array}\]
More concretely, the map $(2,3)$ can be described as follows on a word on the generators $e_{ij}^{(1)},e_{ij}^{(2)}$: first, put all generators $e_{ij}^{(1)}$ to the left; then reverse the order of the generators $e_{ij}^{(2)}$ and replace each one of them by $-e_{ij}^{(1)}-e_{ij}^{(2)}$.
 
Note that the map $(2,3)$ is only linear map on $U(sl(3))\otimes U(sl(3))$ and not an algebra automorphism. However, it turns out that the map $(2,3)$ restricts to an algebra automorphism of the centraliser $Z_2(sl(3))$. We could assume this fact at this point, since we only use the maps $\tau,\,(1,2),\,(2,3)$ as a guide to choose convenient generators, and it will easily follow from the algebraic description of $Z_2(sl(3))$ in Theorem \ref{thm-PT} below.

Nevertheless, it can be checked directly as follows. First, note that $\theta$ intertwines the adjoint actions of $sl(3)$: $\theta([g,x])=[g\otimes 1+1\otimes g,\theta(x)]$ for $x\in U(sl(3))$ and $g\in sl(3)$. It follows that the map $(2,3)$ commutes with the adjoint action on $U(sl(3))\otimes U(sl(3))$, and therefore the map $(2,3)$ restricts to a linear map on $Z_2(sl(3))$. Now to check that it becomes an algebra automorphism, take two elements $X=\sum a\otimes b$ and $Y=\sum c\otimes d$ in $Z_2(sl(3))$. Omitting the sum symbols, we calculate $(2,3)(X)$ times $(2,3)(Y)$:
\[(a\otimes 1).\theta(b). (2,3)(Y)=(a\otimes 1). (2,3)(Y).\theta(b)= (a\otimes 1).(c\otimes 1).\theta(d)\theta(b)=(ac\otimes 1)\theta(bd)\ ,\]
and we find that this is equal to $(2,3)(XY)$. In the last equality, we used that $\theta$ is an antiautomorphism. For the first equality, note that $(2,3)(Y)\in Z_2(sl(3))$ since $Z_2(sl(3))$ is stable by $(2,3)$. And as $\theta(b)$ is in the diagonal embedding of $U(sl(3))$, it therefore commutes with $(2,3)(Y)$.

\medskip
We denote $\textbf{Aut}_0$ the group generated by the three automorphisms $\tau,\,(1,2),\,(2,3)$ of $Z_2(sl(3))$. This group has the following structure:
\[\textbf{Aut}_0=\langle\,(1,2),\,(2,3),\,\tau\,\rangle=\langle\,(1,2),\,(2,3)\,\rangle\times \langle\,\tau\,\rangle \cong S_3\times \mathbb{Z}/2\mathbb{Z}\ ,\]
where $S_3$ is the symmetric group on $3$ letters. The subgroup isomorphic to $S_3$ is generated by $(1,2)$ and $(2,3)$ (which are identified with the simple transpositions) and commutes with $\tau$.

\begin{rem}\label{rem-gen}
There is a natural origin to the automorphisms $(1,2)$ and $(2,3)$. 
Let $U=U(sl(3))^{\otimes 3}$ and let $I=U\delta(sl(3))$ be the left ideal in $U$ generated by all elements $\delta(g)$ with $g\in sl(3)$, where $\delta$ is the diagonal embedding. Using the PBW bases, one can check that:
\[U/I\cong U(sl(3))\otimes U(sl(3))\ \ \ \ \text{(as vector spaces).}\] 
This amounts to say that for an element $x\in U$, there is a unique representative modulo $I$ of the form $\overline{x}\otimes 1$, and thus $\overline{x}$ is the corresponding element in $U(sl(3))\otimes U(sl(3))$. Now, on $U$, there are natural linear maps $(1,2)$ and $(2,3)$ permuting the three copies of $U(sl(3))$ and leaving the ideal $I$ invariant. So they pass to the quotient $U/I$, and through the above isomorphism, to linear maps on $U(sl(3))\otimes U(sl(3))$. It is a straightforward exercise to check that the above definition of $(2,3)$ coincides with this one.

Furthermore, in the above isomorphism of vector spaces, one can take on both sides the invariants with respect to the adjoint action of $sl(3)$. This way one gets an isomorphism:
\[\left(U/I\right)^{sl_3}\cong Z_2(sl(3))\ .\]
The left-hand-side is now an algebra, being an example of quantum Hamiltonian reduction (see for example \cite[\S1.1]{EGGO}). Thus the isomorphism becomes an isomorphism of algebras and the linear maps $(1,2)$ and $(2,3)$ descend to algebras automorphisms.
\end{rem}

\subsection{Choice of generators of $Z_2(sl(3))$.} First we analyse the action of $\textbf{Aut}_0$ on the small degree components  of $Z_2(sl(3))$. Straightforward calculations give the matrices representing the action of $\tau,\, (1,2),\,(2,3)$ on the basis $T^{(1,1)},\,T^{(1,2)},\,T^{(2,2)},\,T^{(1,1,1)},\,T^{(1,1,2)},\,T^{(1,2,2)},\,T^{(2,2,2)}$ of polarised traces of degree 2 and 3 (the dots are zeros):
\[\left(\begin{array}{ccccccc}
1 & \cdot & \cdot & -3 & \cdot & \cdot & \cdot\\
\cdot & 1 & \cdot & \cdot & -3 & -3 & \cdot\\
\cdot & \cdot & 1 & \cdot & \cdot & \cdot & -3\\
\cdot & \cdot & \cdot & -1 & \cdot & \cdot & \cdot\\
\cdot & \cdot & \cdot & \cdot & -1 & \cdot & \cdot\\
\cdot & \cdot & \cdot & \cdot & \cdot & -1 & \cdot\\
\cdot & \cdot & \cdot & \cdot & \cdot & \cdot & -1
\end{array}\right)\ \ \ 
\left(\begin{array}{ccccccc}
\cdot & \cdot & 1 & \cdot & \cdot & \cdot & \cdot\\
\cdot & 1 & \cdot & \cdot & \cdot & \cdot & \cdot\\
1 & \cdot & \cdot & \cdot & \cdot & \cdot & \cdot\\
\cdot & \cdot & \cdot & \cdot & \cdot & \cdot & 1\\
\cdot & \cdot & \cdot & \cdot & \cdot & 1 & \cdot\\
\cdot & \cdot & \cdot & \cdot & 1 & \cdot & \cdot\\
\cdot & \cdot & \cdot & 1 & \cdot & \cdot & \cdot
\end{array}\right)\ \ \ 
\left(\begin{array}{ccccccc}
1 & -1 & 1 & \cdot & \cdot & 3 & -3\\
\cdot & -1 & 2 & \cdot & \cdot & 6 & -12\\
\cdot & \cdot & 1 & \cdot & \cdot & \cdot & -3\\
\cdot & \cdot & \cdot & 1 & -1 & 1 & -1\\
\cdot & \cdot & \cdot & \cdot & -1 & 2 & -3\\
\cdot & \cdot & \cdot & \cdot & \cdot & 1 & -3\\
\cdot & \cdot & \cdot & \cdot & \cdot & \cdot & -1
\end{array}\right)\,.\]
For degree 4, in addition to the product of two polarised traces of degree 2, we have $T^{(1,1,2,2)}$ and the automorphisms transform it as follows:
\[\begin{array}{rl}
\tau & :\ T^{(1,1,2,2)}\mapsto\ T^{(1,1,2,2)}+3(T^{(1,1,2)}+T^{(1,2,2)})+9 T^{(1,2)}\,,\\[0.4em]
(1,2) & :\ T^{(1,1,2,2)}\mapsto\ T^{(1,1,2,2)}\,,\\[0.4em]
(2,3) & :\ T^{(1,1,2,2)}\mapsto\ T^{(1,1,2,2)}+T^{(1,1)}T^{(1,2)}+\frac{1}{2}(T^{(1,1)})^2-4 T^{(1,2)}-2 T^{(1,1)}\ .\end{array}\]
By direct inspection, we find the following representations of the automorphism group $\textbf{Aut}_0$:
\begin{itemize}
\item In degree 2, $S_3$ acts through its natural permutation representation and $\tau$ acts by the identity.
\item Including degree 3, it adds for $S_3$ a copy of the permutation representation and a copy of the sign representation; for $\tau$, it adds four times the eigenvalue $-1$.
\item Then, bringing in degree 4, it adds for $S_3$ two copies of the permutation representation (coming from the symmetric square of the representation in degree 2) and a copy of the trivial representation; for $\tau$, it adds only trivial representations.
\end{itemize}
Therefore, it is natural to take for generators of $Z_2(sl(3))$ a set of elements:
\[k_1,k_2,k_3,\ l_1,l_2,l_3,\ X,\ Y\ ,\]
such that the degrees are respectively $2,2,2,3,3,3,3,4$, and moreover such that the action of $\textbf{Aut}_0$ is as follows:
\begin{equation}\label{actS3}
\pi\in S_3\ :\ \ \ \ \ k_i\mapsto k_{\pi(i)}\,,\ \ \ \ l_i\mapsto l_{\pi(i)}\,,\ \ \ \ X\mapsto sgn(\pi)X\,,\ \ \ \ Y\mapsto Y\,,
\end{equation}
\begin{equation}\label{acttau}
\tau\ :\ \ \ \ \ k_i\mapsto k_{i}\,,\ \ \ \ l_i\mapsto -l_{i}\,,\ \ \ \ X\mapsto -X\,,\ \ \ \ Y\mapsto Y\,.
\end{equation}
\begin{rem}
Up to some normalisation and trivial addition of constants, these requirements uniquely fix the choice of the generator $X$. They fix the choice of $Y$ up to the addition of a symmetric polynomial of degree 2 in $k_1,k_2,k_3$. The choice for $Y$ we make is such that $[X,[X,Y]]$ contains no linear term in $Y$.
\end{rem}
We will give now an explicit expression for our generators. First, recall a standard definition of the Casimir elements of $U(sl(3))$:
\[
 \overline C^{(2)}=  e_{i_1 i_2} e_{i_2 i_1}\ \ \ \text{and}\ \ \ \overline C^{(3)}=  e_{i_1 i_2} e_{i_2 i_3}e_{i_3,i_1}\ .
\]
We will use also a slightly different version:
\begin{equation}\label{cas2cas3}
 C^{(2)}=  e_{i_2 i_1} e_{i_1 i_2}\ \ \ \text{and}\ \ \ C^{(3)}=  e_{i_2 i_1} e_{i_3 i_2}e_{i_1,i_3}\ .
\end{equation}
The relation between the two versions in $U(sl(3))$ is:
\begin{equation}
C^{(2)}= \overline C^{(2)}\ ,\quad C^{(3)}= \overline C^{(3)}-3\overline{C}^{(2)}\ .
\end{equation}
With these notations, we set:
\begin{equation}\label{defgen}
\begin{array}{l}
k_1=C^{(2)}\otimes 1\,,\ \ k_2=1\otimes C^{(2)}\,,\ \ k_3=\delta(C^{(2)})\,,\\[0.5em]
\displaystyle l_1=\frac{1}{2}(C^{(3)}+\overline C^{(3)})\otimes 1\,,\ \ l_2=1\otimes \frac{1}{2}(C^{(3)}+\overline C^{(3)})\,,\ \ l_3=-\delta(\frac{1}{2}(C^{(3)}+\overline C^{(3)}))\,,\\[0.8em]
\displaystyle X=\frac{1}{2}(T^{(1,1,2)}-T^{(1,2,2)})+\frac{1}{3}(l_1-l_2)\,,\\[0.8em]
\displaystyle Y=T^{(1,1,2,2)}+\frac{3}{2}(T^{(1,1,2)}+T^{(1,2,2)})-\frac{1}{12}(T^{(1,2)})^2-\frac{5}{12}T^{(1,1)}T^{(2,2)}+\frac{5}{2}T^{(1,2)}\,,\\[0.8em]
\displaystyle Z=[X,Y]\ .
\end{array}
\end{equation} 
\begin{prop}
The elements $k_1,k_2,k_3,l_1,l_2,l_3,X,Y$ given by (\ref{defgen}) are generators of $Z_2(sl(3))$ satisfying the symmetry conditions (\ref{actS3})-(\ref{acttau}).
\end{prop}
\begin{proof}
Everything can be checked easily by straightforward calculations.
\end{proof}
\begin{rem}
In terms of polarised traces, it is not difficult to verify that:
\[
\begin{array}{l}
k_1=T^{(1,1)}\,,\ \ k_2=T^{(2,2)}\,,\ \ k_3=T^{(1,1)}+T^{(2,2)}+2T^{(1,2)}\,,\\[0.5em]
\displaystyle l_1=T^{(1,1,1)}+\frac{3}{2}k_1\,,\ \ l_2=T^{(2,2,2)}+\frac{3}{2}k_2\,,\ \ l_3=-l_1-l_2-3(T^{(1,1,2)}+T^{(1,2,2)})-9T^{(1,2)}\,.
\end{array}
\]
\end{rem}

\subsection{A realisation of $\alg$ in $U(sl(3))\otimes U(sl(3))$} 

We start by introducing a particular case $\mathcal{A}$ of the algebra $\alg$ of Section \ref{sec-A}, which will turn out to be the one relevant for the description of $Z_2(sl(3))$.

In the definition below, we consider $k_1,k_2,k_3,l_1,l_2,l_3$ as indeterminates (central parameters). The symmetric group $S_3$ acts on $\mathbb{C}[k_1,k_2,k_3,l_1,l_2,l_3]$ by permuting $\{k_1,k_2,k_3\}$ and permuting $\{l_1,l_2,l_3\}$. We denote:
\[Sym^+:=\frac{1}{6}\sum_{\pi\in S_3}\pi\ \ \ \ \ \text{and}\ \ \ \ \ Sym^-:=\frac{1}{6}\sum_{\pi\in S_3}sgn(\pi)\pi\ .\]
\begin{definition}\label{def:cA2}
The algebra $\mathcal{A}$ is the algebra over $\mathbb{C}[k_1,k_2,k_3,l_1,l_2,l_3]$ generated by $A,B,C$ with the following defining relations:
\begin{equation}\label{relAlg2}
\begin{array}{l}
[A,B]=C\,,\\[0.5em]
[A,C]= -6B^2 + a_2 A^2+ a_5 A+a_8\,,\\[0.5em]
[B,C]= -2A^3-a_2\{A,B\} -a_5 B+a_6A+a_9\,,
\end{array}
\end{equation}
where:
\[a_2=\frac{1}{2}(k_1+k_2+k_3)\,,\ \ \ \ \ \ a_5=Sym^-(2k_2l_1)\ ,\]
\[\begin{array}{rl}
a_6=Sym^+\Bigl(
 & \frac{1}{16}(-k_1^3+2k_1^2k_2-6k_1k_2k_3)+\frac{3}{8}(k_1^2-2k_1k_2)+\frac{1}{6}l_1^2-\frac{5}{3}l_1l_2
\Bigr)\end{array},\]
\[\begin{array}{rl}
a_8=Sym^+\Bigl( & 
\frac{1}{128}(k_1^4-8k_1^3k_2+6k_1^2k_2^2+4k_1^2k_2k_3)-\frac{1}{32}(3k_1^3-6k_1^2k_2+2k_1k_2k_3)\\
 & -\frac{1}{24}k_1(l_1^2-6l_2^2+4l_1l_2-26l_2l_3)+\frac{1}{2}l_1^2+l_1l_2
\Bigr)\end{array}\]
\[\begin{array}{rl}
a_9=Sym^-\Bigl(
 & l_2\bigl(-\frac{4}{9}l_1^2+\frac{1}{24}(k_1^3-2k_1^2k_2+k_1k_2^2+9k_1^2k_3)+\frac{1}{2}(k_1^2-k_1k_2)\bigr)\Bigr)\end{array},\]
\end{definition}
Clearly, the algebra $\mathcal{A}$ can be seen as a specialisation of the generic algebra $\alg$ and as such has a PBW basis as explained in Section \ref{sec-A}. The use of the same names as some generators of $Z_2(sl(3))$ is a convenient slight abuse of notation and is justified by the next proposition. Below, it is understood that through $\phi$, the parameters $k_1,k_2,k_3,l_1,l_2,l_3$ are sent to the generators sharing the same names.
\begin{prop}\label{prop-morA}
The following map $\phi$ extends to a surjective morphism of algebras from $\mathcal{A}$ to $Z_2(sl(3))$:
\begin{equation}\label{eq-morA}
\phi\ :\ \ \ \ A\mapsto X\,,\qquad B\mapsto Y\,,\qquad C\mapsto Z\,.
\end{equation}
\end{prop}
\begin{proof}
We first note that the elements $k_1,k_2,k_3,l_1,l_2,l_3$ are all central in $Z_2(sl(3))$. 
Indeed, $k_1,k_2,l_1,l_2$ are already central in $U(sl(3))\otimes U(sl(3))$ since they are of the form $c\otimes 1$ or $1\otimes c$ with $c$ a central element of $U(sl(3))$. 
Concerning $k_3$ and $l_3$, we see from their definitions (\ref{defgen}) that they are of the form $\delta(c)$ with $c\in U(sl(3))$. Those elements $\delta(c)$ being in the image of the diagonal map commute with all elements of $Z_2(sl(3))$. 

It remains to calculate in $U(sl(3))\otimes U(sl(3))$ the commutators $[X,Z]$ and $[Y,Z]$ to see that they satisfy the defining relations of $\alg$ when the images of the central elements are given as in the statement of the Proposition. This is a straightforward but very lengthy calculation which can be done with the help of a computer, 
by rewriting all elements involved in the PBW basis of $U(sl(3))\otimes U(sl(3))$, and comparing both sides of the given relations.
\end{proof}
\begin{rem} In the explicit realisation of $\mathcal{A}$ in $U(sl(3))\otimes U(sl(3))$ given above, the images of the generators of $\mathcal{A}$ (as an algebra over $\mathbb{C}$) have the following degrees:
\begin{equation}\label{deg-A}
\begin{array}{c}
deg(X)=3\,,\ \ \ \ \ deg(Y)=4\,,\ \ \ \ \ (\text{and thus } deg(Z)=6)\,,\ \ \ \ \deg(a_i)=i\ .
\end{array}
\end{equation}
It is remarkable that the degrees of the non vanishing central elements $a_i$ are 2,5,6,8 and 9 which are exactly the first five fundamental degrees of type $E_6$. 
The missing degree is 12, but such a degree 12 polynomial will also appear naturally in the complete description of $Z_2(sl(3))$ (see next subsection). 
This ``coincidence'' will be largely developed in the next section.
\end{rem}
\begin{rem} The parameters $a_1,a_3$ and $a_4$ of the generic algebra $\alg$ are sent to $0$ in the specialisation in $\mathcal{A}$ and therefore also in $Z_2(sl(3))$. 
However, a different choice of generators $X,Y$ would result in non-zero values for $a_3$ and $a_4$. 
Let us also point out that a non-zero coefficient $a_1$ would also appear if we were looking at $Z_2(gl(3))$.
\end{rem}

\begin{rem}
 As mentioned in the Introduction, another specialisations of the algebra $\alg$ has been used in \cite{racah} in a study of the centraliser of $so(3)$ in $su(3)$.
\end{rem}

\subsection{The complete description of $Z_2(sl(3))$.}\label{subsec-pres}
At this point, we know that the algebra $Z_2(sl(3))$ is isomorphic to a quotient of the algebra $\mathcal{A}$. Our goal now is to complete the description of $Z_2(sl(3))$ by identifying the kernel of the map $\phi$.

To guide us, we have on one side the PBW basis of the algebra $\mathcal{A}$, which asserts that as a vector space, this is a polynomial algebra. And on the other side, we have the Hilbert--Poincar\'e series of $Z_2(sl(3))$ given in (\ref{HP-sl3}). 
The comparison of both descriptions tells us that we must have one more relation satisfied by the images of the generators of $\mathcal{A}$ in $Z_2(sl(3))$. 
Moreover, this relation has to be found at the degree 12 (or more precisely at degree (6,6) in $U(sl(3))\otimes U(sl(3))$).

Recall that the algebra $\mathcal{A}$ inherits from $\alg$ a Casimir element, which was given in (\ref{eq:CAP}). 
In the specialisation $\mathcal{A}$, its expression simplifies to:
\begin{equation}\label{eq:CAP2}
\Omega =x_1A+x_2B+x_3A^2+x_4\{A,B\}
+x_5 B^2+x_7 ABA-A^4+4B^3+C^2
\end{equation}
where (with the parameters $a_2,a_5,a_6,a_8,a_9$ given in Definition \ref{def:cA2}):
\begin{equation}
 x_1= 6 a_5  + 2 a_9\,,\ \ \ x_2= -2a_6 - 2 a_8\,,\ \ \ x_3= 6 a_2 + a_6\,,\ \ \ x_4=-a_5\,,\ \ \ x_5= 8 a_2 -24\,,\ \ \ x_7= -2 a_2 +12\ . \label{eq:xx}
\end{equation}

Remarkably, the image of this central element $\Omega$ in $Z_2(sl(3))$ is of degree 12 and it turns out that we have in $Z_2(sl(3))$:
\[\phi(\Omega)=a_{12}(k_1,k_2,k_3,l_1,l_2,l_3)\,,\]
where the polynomial $a_{12}(k_1,k_2,k_3,l_1,l_2,l_3)$ is equal to:
\begin{equation}\label{eq:a12}
\begin{array}{l}
Sym^+\Bigl( \frac{1}{4608}k_1^6 - \frac{1}{768}k_1^5 (1 + 2 k_2) + 
\frac{1}{768} k_1^4 (-39 + 6 k_2 + 5 k_2^2 + 3 k_2 k_3)-\frac{1}{1152}k_1^3 (-648 - 468 k_2 \\[0.5em]
\quad + 6 k_2^2 + 5 k_2^3 - 6 k_2 k_3 + 6 k_2^2 k_3 + 
    2 l_1^2 + 8 l_1 l_2 - 12 l_2^2 + 140 l_2 l_3) - \frac{1}{2304}k_1^2 (2592 k_2 + 702 k_2^2 \\[0.5em]
\quad + 468 k_2 k_3 + 42 k_2^2 k_3 + k_2^2 k_3^2 - 
    40 l_1^2 - 40 k_2 l_1^2 + 416 l_1 l_2 - 16 k_2 l_1 l_2 + 144 k_3 l_1 l_2 - 464 l_2^2 \\[0.4em]
\quad + 56 k_2 l_2^2 - 360 k_3 l_2^2 - 560 l_2 l_3 - 
    720 k_3 l_2 l_3)  -\frac{1}{288} k_1 (-108 k_2 k_3 - 24 l_1^2 + 68 k_2 l_1^2 + 17 k_2 k_3 l_1^2  \\[0.4em]
\quad - 672 l_1 l_2 - 52 k_2 l_1 l_2 - 296 k_3 l_1 l_2 - 2 k_2 k_3 l_1 l_2 - 48 l_2^2 - 206 k_3 l_2^2 + 1392 l_2 l_3)
    + \frac{1}{432}l_1 (-864 l_1 \\[0.4em]
\quad + l_1^3 - 1728 l_2 + 24 l_1^2 l_2 - 26 l_1 l_2^2 +  244 l_1 l_2 l_3)
\Bigr)\,.
\end{array}
\end{equation}
This formula is checked by a straightforward calculation directly in $U(sl(3))\otimes U(sl(3))$. This can be done using a computer to express both sides in a PBW basis of $U(sl(3))\otimes U(sl(3))$.
\begin{thm}\label{thm-PT}
The algebra $Z_2(sl(3))$ is the quotient of the algebra $\mathcal{A}$ by the following additional relation:
\begin{equation}\label{imcas}
\Omega=a_{12}(k_1,k_2,k_3,l_1,l_2,l_3)\ .
\end{equation}
\end{thm}
\begin{proof}
Recall that we have the surjective morphism $\phi$ from Proposition \ref{prop-morA} from $\mathcal{A}$ to $Z_2(sl(3))$. Moreover, as already said, the relation (\ref{imcas}) is satisfied by the image of $\Omega$ in $U(sl(3))\otimes U(sl(3))$. So the element $\Omega-a_{12}(k_1,k_2,k_3,l_1,l_2,l_3)$ does belong to the kernel of the morphism $\phi$.

Both algebras $\mathcal{A}$ and $Z_2(sl(3))$ are filtered complex algebras. The degrees in $\mathcal{A}$ are given by:
\[deg(k_i)=2\,,\ \ \ deg(l_i)=deg(A)=3\,,\ \ deg(B)=4\,,\ \ deg(C)=6\,,\]
and the degrees in $Z_2(sl(3))$ come naturally from the degrees in $U(sl(3))\otimes U(sl(3))$ (the generators $e_{ij}^{(a)}$ are of degree 1). 
Then it is immediate that $\phi$ is a morphism of filtered algebras (the degrees are preserved).

Thus it remains only to show that the Hilbert--Poincar\'e series of the quotient of $\mathcal{A}$ by the additional relation (\ref{imcas}) 
coincides with the Hilbert--Poincar\'e series of $Z_2(sl(3))$. From the PBW basis of $\mathcal{A}$, we have that its Hilbert--Poincar\'e series in one variable is:
\[
\frac{1}{(1-t^2)^3(1-t^3)^4(1-t^4)(1-t^6)}\ .
\]
Adding the relation (\ref{imcas}) has the effect of multiplying the numerator by a factor $(1-t^{12})$. Indeed, the relation expresses $C^2$ in the form $x+yC$ where $x,y$ involves only the other generators. Thus, using the PBW basis of $\mathcal{A}$, it is easy to see that a basis of the quotiented algebra consists of all ordered monomials in $k_1,k_2,k_3,l_1,l_2,l_3,A,B,C$ with the restriction that the power of $C$ must be $0$ or $1$.

So we end up with the Hilbert--Poincar\'e series of $Z_2(sl(3))$ given in (\ref{HP-sl3}) when the bidegree is reduced to a single degree. The proof is concluded.
\end{proof}

\section{Highest-weight specialisation of $Z_2(sl(3))$ and $E_6$ symmetry}

\subsection{A highest-weight specialisation of $\mathcal{A}$.}
Consider a highest weight representation $V_{m_1,m_2}$ of $U(sl(3))$. It is parametrised by two complex numbers $m_1,m_2$, and generated as a $U(sl(3))$-module by a highest weight vector $v_{m_1,m_2}$ satisfying
\begin{equation}\label{high-weight-rep}
\begin{array}{rcl} h_{p}v_{m_1,m_2}&=& (m_p-1) v_{m_1,m_2} \quad \text{with} \quad p=1,2\,,\\[0.5em]
 e_{pq}v_{m_1,m_2}&=&0 \quad \text{with} \quad 1\leq p< q \leq 3 \,,
\end{array}
\end{equation}
where we set $h_p=e_{pp}-e_{p+1,p+1}$.
\begin{rem}
The standard convention would denote $(m_1-1,m_2-1)$ the highest weight of the representation $V_{m_1,m_2}$ since $m_1-1$ and $m_2-1$ are the coefficients in front of the fundamental weights. 
For example, if $m_1-1,m_2-1\in\mathbb{Z}_{\geq 0}$ and $V_{m_1,m_2}$ is the corresponding finite-dimensional irreducible representation of $sl(3)$, then it is usually represented with a Young diagram with $m_1+m_2-2$ boxes in the first row and $m_2-1$ boxes in the second row.

We find our parametrisation more convenient for the purpose of this paper since it includes from the beginning the translation by the half sum of the positive roots of $sl(3)$.
\end{rem}
 
In the representation $V_{m_1,m_2}$, the Casimir elements $C^{(2)}$ and $C^{(3)}$ of $U(sl(3))$ defined in (\ref{cas2cas3}) are proportional to the identity matrix, their values are given by:
\begin{eqnarray}
c^{(2)}(m_1,m_2)&=&\frac{2}{3}( m_1^2+ m_2^2+ m_1 m_2)-2\,, \label{value-Cas1}\\
c^{(3)}(m_1,m_2)&=&\frac{1}{9}(m_1+2m_2-3)(2m_1+m_2+3)(m_1-m_2-3)\ . \label{value-Cas2}
\end{eqnarray}

Now pick three pairs of complex numbers $(m_1,m_2)$, $(m'_1,m'_2)$ and $(m''_1,m''_2)$, and look at the tensor product of representations $V_{m_1,m_2}\otimes V_{m'_1,m'_2}$. In this space, consider all the highest-weight vectors under the diagonal action of $U(sl(3))$ corresponding to $(m''_1,m''_2)$, that is, all vectors satisfying for the diagonal action Formulas (\ref{high-weight-rep}) with $m_1,m_2$ replaced by $m''_1,m''_2$. 
The subspace generated by all these vectors as a $U(sl(3))$-module is denoted by:
\[M_{m_1,m_2,m_1',m_2'}^{m''_1,m''_2}\ .\]
\begin{ex}\label{exM}
In particular, if the three pairs are integers in $\mathbb{Z}_{\geq 1}$ and $V_{m_1,m_2},V_{m'_1,m'_2},V_{m''_1,m''_2}$ are the corresponding finite-dimensional irreducible representations, then the subspace $M_{m_1,m_2,m_1',m_2'}^{m''_1,m''_2}$ is the isotypic component of $V_{m''_1,m''_2}$ in the tensor product $V_{m_1,m_2}\otimes V_{m'_1,m'_2}$. In this case, we have:
\[
M_{m_1,m_2,m_1',m_2'}^{m''_1,m''_2}=D_{m_1,m_2,m_1',m_2'}^{m''_1,m''_2}V_{m''_1,m''_2}\,,
\]
where the positive integer $D_{m_1,m_2,m_1',m_2'}^{m''_1,m''_2}$ is the multiplicity of $V_{m''_1,m''_2}$ in $V_{m_1,m_2}\otimes V_{m'_1,m'_2}$ which is called the Littlewood--Richardson coefficient.
\end{ex}
As a subalgebra of $U(sl(3))\otimes U(sl(3))$, the diagonal centraliser $Z_2(sl(3))$ acts on $V_{m_1,m_2}\otimes V_{m'_1,m'_2}$. 
Moreover, since it centralises the diagonal action of $U(sl(3))$, it leaves invariant the subspace $M_{m_1,m_2,m_1',m_2'}^{m''_1,m''_2}$, which thus becomes naturally a $Z_2(sl(3))$-module.

In this representation $M_{m_1,m_2,m_1',m_2'}^{m''_1,m''_2}$ of $Z_2(sl(3))$, the central parameters $k_1,k_2,k_3,l_1,l_2,l_3$ take definite complex values, since they are expressed in terms of the Casimir elements. Indeed recall that they were defined as:
\[\begin{array}{l}
k_1=C^{(2)}\otimes 1\,,\ \ k_2=1\otimes C^{(2)}\,,\ \ k_3=\delta(C^{(2)})\,,\\[0.5em]
\displaystyle l_1=(C^{(3)}+\frac{3}{2}C^{(2)})\otimes 1\,,\ \ l_2=1\otimes (C^{(3)}+\frac{3}{2} C^{(2)})\,,\ \ l_3=-\delta(C^{(3)}+\frac{3}{2}C^{(2)})\,.
\end{array}
\]
This motivates the following definition, which entails a specialisation of $\mathcal{A}$ and $Z_2(sl(3))$ acting on the space $M_{m_1,m_2,m_1',m_2'}^{m''_1,m''_2}$.
\begin{definition}\label{defAspec}
The algebra $\mathcal{A}^{spec}$ is the specialisation of $\mathcal{A}$ corresponding to the following values of the central parameters:
\[\begin{array}{c}
k_1=c^{(2)}(m_1,m_2)\,,\ \ k_2=c^{(2)}(m'_1,m'_2)\,,\ \ k_3=c^{(2)}(m''_1,m''_2)\,,\\[0.5em]
\displaystyle l_1=c^{(3)}(m_1,m_2)+\frac{3}{2}c^{(2)}(m_1,m_2)\,,\ \ l_2=c^{(3)}(m'_1,m'_2)+\frac{3}{2}c^{(2)}(m'_1,m'_2)\,,\\[0.8em]
\displaystyle l_3=-c^{(3)}(m''_1,m''_2)-\frac{3}{2}c^{(2)}(m''_1,m''_2)\,.
\end{array}
\]
Similarly, we denote by $Z_2(sl(3))^{spec}$ the same specialisation of $Z_2(sl(3))$.
\end{definition}
In simple terms, this specialisation means that in Definition \ref{def:cA2} of $\mathcal{A}$ and in the description of $Z_2(sl(3))$ in Theorem \ref{thm-PT}, 
the central elements $k_1,k_2,k_3,l_1,l_2,l_3$ are replaced by their expressions above in terms of $m_1,m_2,m'_1,m'_2,m''_1,m''_2$. 
Note that these specialisations are well-defined since both $\mathcal{A}$ and $Z_2(sl(3))$ are free modules over $\mathbb{C}[k_1,k_2,k_3,l_1,l_2,l_3]$.
In the algebra  $\mathcal{A}^{spec}$, the parameter $a_i$ becomes a polynomial in $m_1$, $m_2$, $m'_1$, $m'_2$, $m''_1$ and  $m''_2$ of degree $i$. 
Similarly $a_{12}$ becomes a polynomial of degree $12$ in  $Z_2(sl(3))^{spec}$.

We can summarise the results obtained so far in the following proposition:
\begin{prop} \label{pr:Aspec}
The algebra $Z_2(sl(3))^{spec}$ is generated by $X,Y,Z$ with the defining relations:
\begin{equation}
\begin{array}{l}
[X,Y]=Z\,,\\[0.5em]
[X,Z]= -6Y^2 + a_2 X^2+ a_5 X+a_8\,,\\[0.5em]
[Y,Z]=-2X^3-a_2\{X,Y \} -a_5 Y+a_6X+a_9\,,\\[0.5em]
x_1X+x_2Y+x_3X^2+x_4\{X,Y\} +x_5 Y^2+x_7 XYX-X^4+4Y^3+Z^2=a_{12}\,,
\end{array}
\end{equation}
where $x_i$ are given by \eqref{eq:xx} and $a_i$ are given in Definition \ref{def:cA2} and by \eqref{eq:a12} in which we replace $k_1,k_2,k_3,l_1,l_2,l_3$ 
 by their expressions in terms of $m_1,m_2,m'_1,m'_2,m''_1,m''_2$ given in Definition  \ref{defAspec}.
\end{prop}

\begin{rem}
Consider the situation of Example \ref{exM} and denote by $\cX$ the endomorphism giving the action of the generator $X$ on the highest weight vectors in $M_{m_1,m_2,m_1',m_2'}^{m''_1,m''_2}$. The matrix $\cX$ has somehow been studied previously in \cite{PST} where a matrix $S$ related to $\cX$ through $S=-54\cX$ is introduced.
It has been shown in \cite{PST} that the eigenvalues of $S$ are non-degenerate and allow to distinguish the different highest-weight vectors in $M_{m_1,m_2,m_1',m_2'}^{m''_1,m''_2}$ thereby providing a solution to the missing label problem for the tensor 
product of two irreducible representations of $sl(3)$.
Let us also mention \cite{CaSt} where an operator $\Theta^{[0,3]}= -2\cX+l_1-l_2$ has been studied in the case where the 
Littlewood--Richardson coefficient $D_{m_1,m_2,m_1',m_2'}^{m''_1,m''_2}$ is 2.
\end{rem}

The explicit expressions of the polynomials $a_i$ in terms of $m_1,m_2,m'_1,m'_2,m''_1,m''_2$ are quite complicated but it turns out that these polynomials 
are related to the fundamental invariant polynomials of the Weyl group of type $E_6$. A first hint in this direction comes from observing that the degrees of the polynomials in Proposition \ref{pr:Aspec}
are $2,5,6,8,9$ and $12$ which are the exactly the degrees of the fundamental invariant polynomials of $E_6$.
We shall explain that in detail in the next subsection.


\subsection{The Weyl group of type $E_6$.}

Let us consider a root system of type $E_6$ and choose the simple roots $\alpha_1,\alpha_2,\alpha_3,\alpha_4,\alpha_5,\alpha_6$ with the numeration accordingly to the following Dynkin diagram:
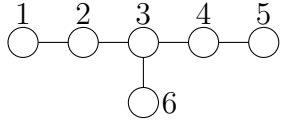
\begin{center}
\begin{tikzpicture}[scale=0.2]
\draw (1,1) circle [radius=1];
\node [above] at (1,1.5) {$1$};
\draw (2,1)--(4,1);

\draw (5,1) circle [radius=1];
\node [above] at (5,1.5) {$2$};
\draw (6,1)--(8,1);

\draw (9,1) circle [radius=1];
\node [above] at (9,1.5) {$3$};
\draw (10,1)--(12,1);

\draw (13,1) circle [radius=1];
\node [above] at (13,1.5) {$4$};
\draw (14,1)--(16,1);

\draw (17,1) circle [radius=1];
\node [above] at (17,1.5) {$5$};

\draw (9,-3) circle [radius=1];
\node [right] at (9.5,-3) {$6$};
\draw (9,0)--(9,-2);

\end{tikzpicture}
\end{center}
We denote by $W(E_6)$ the corresponding Weyl group and by $s_i=s_{\alpha_i}$ the Weyl reflections associated to the roots $\alpha_i$ satisfying, for $1\leq i,j\leq 6$,
\begin{eqnarray}
 &&s_i^2=1\\
 &&s_is_j=s_js_i \qquad \text{if $i$ and $j$ are not connected in the Dynkin diagram}\\
 &&s_is_js_i=s_js_is_j  \qquad \text{if $i$ and $j$ are connected in the Dynkin diagram}.
\end{eqnarray}
Let us associate the parameters $m_1,m_2,m'_1,m'_2,m''_1$ and $m''_2$ with the simple roots as follows:
\begin{equation}\label{roots}
\begin{array}{l}
m_1=\alpha_1\,,\\[0.4em]
m_2=\alpha_2\,,
\end{array}\ \ \ \ \ \ \ \begin{array}{l}
m'_1=\alpha_5\,,\\[0.4em]
m'_2=\alpha_4\,,
\end{array}\ \ \ \ \ \ \ \
\begin{array}{l}
m''_1=\Theta\,,\\[0.4em]
m''_2=-\alpha_6\,,
\end{array}\end{equation}
where $\Theta$ is the longest positive root of $E_6$. The explicit expression of $\Theta$ is such that:
\[\alpha_3=\frac{1}{3}\bigl(m''_1+2m''_2-(m_1+2m_2)-(m'_1+2m'_2)\bigr)\ .\]
Note that $(m_1,m_2)$, $(m'_1,m'_2)$ and $(m''_1,m''_2)$ are three subsystems of type $A_2$ which are pairwise orthogonal. In other words, they generate a subsystem of type $A_2\times A_2\times A_2$.

The Weyl group acts on the set of roots by its usual reflection representation, and it is elementary to calculate the induced action of the generators of $W(E_6)$ on the parameters. We have:
\begin{equation}\label{eq:s1}
s_1\ :\left\{\begin{array}{l}
m_1\mapsto -m_1\,,\\[0.4em]
m_2\mapsto m_1+m_2\,,
\end{array}\right.\ \ \ 
s_2\ :\left\{\begin{array}{l}
m_1\mapsto m_1+m_2\,,\\[0.4em]
m_2\mapsto -m_2\,,
\end{array}\right.\end{equation}
\begin{equation}\label{eq:s2} 
s_5\ :\left\{\begin{array}{l}
m'_1\mapsto -m'_1\,,\\[0.4em]
m'_2\mapsto m'_1+m'_2\,,
\end{array}\right.
\ \ \ s_4\ :\left\{\begin{array}{l}
m'_1\mapsto m'_1+m'_2\,,\\[0.4em]
m'_2\mapsto -m'_2\,,
\end{array}\right.\end{equation}
\begin{equation}\label{eq:s3}
s_6\ :\left\{\begin{array}{l}
m''_1\mapsto m''_1+m''_2\,,\\[0.4em]
m''_2\mapsto -m''_2\,,
\end{array}\right.\ \ \ 
s_3\ :\left\{\begin{array}{l}
m_2\mapsto m_2+\alpha_3\,,\\[0.4em]
m'_2\mapsto m'_2+\alpha_3\,,\\[0.4em]
m''_2\mapsto m''_2-\alpha_3\,,
\end{array}\right.\end{equation}
where the actions that are not given are trivial and the explicit expression of $\alpha_3$ is provided above.

\vskip .2cm
It is elementary that linear transformations on roots can alternatively be seen as linear transformations of coordinates in the dual space, which is here the space of weights since $E_6$ is simply-laced. To be explicit, consider the weight lattice of type $E_6$, that is:
\[P_{E_6}=\mathbb{Z}\omega_1\oplus \mathbb{Z}\omega_2 \oplus \mathbb{Z}\omega_3 \oplus \mathbb{Z}\omega_4\oplus \mathbb{Z}\omega_5 \oplus \mathbb{Z}\omega_6\,,\]
where $\omega_1,\omega_2,\omega_3,\omega_4,\omega_5,\omega_6$ are the fundamental weights corresponding to the simple roots. 
It comes naturally with an action of the Weyl group $W(E_6)$, which reads explicitly as:
\[s_1(\omega_1)=-\omega_1+\omega_2\,,\ \ \ s_2(\omega_2)=-\omega_2+\omega_1+\omega_3\,,\ \ \ s_3(\omega_3)=-\omega_3+\omega_2+\omega_4+\omega_6\,,\]
\[s_4(\omega_4)=-\omega_4+\omega_3+\omega_5\,,\ \ \ \ s_5(\omega_5)=-\omega_5+\omega_4\,,\ \ \ \ s_6(\omega_6)=-\omega_6+\omega_3\ ,\]
and the unspecified actions are trivial since $s_i(\omega_j)=\omega_j$ if $i\neq j$.

Now, we parametrise a vector of the space spanned by $\omega_1,\omega_2,\omega_3,\omega_4,\omega_5,\omega_6$ as follows:
\[\omega=m_1\omega_1+m_2\omega_2+\frac{1}{3}\bigl(m''_1+2m''_2-(m_1+2m_2)-(m'_1+2m'_2)\bigr)\omega_3+m'_2\omega_4+m'_1\omega_5-m''_2\omega_6\ .\]
In other words, we define $(m_1,m_2,m'_1,m'_2,m''_1,m''_2)$ as the coordinates in the new basis:
\[(\omega_1-\frac{1}{3}\omega_3,\ \omega_2-\frac{2}{3}\omega_3,\ \omega_5-\frac{1}{3}\omega_3,\ \omega_4-\frac{2}{3}\omega_3,\ \frac{1}{3}\omega_3,\ -\omega_6+\frac{2}{3}\omega_3)\ .\] 
Then, it is immediate that the action of the Weyl group on weights transforms the parameters $m_1,m_2,m'_1,m'_2,m''_1,m''_2$ as per the formulas (\ref{eq:s1})--(\ref{eq:s3}).

\subsection{Fundamental invariants and $E_6$ symmetry of $Z_2(sl(3))$.}

Consider now the polynomial algebra $S(V)$ associated to the reflection representation of $W(E_6)$. The roots associated to the parameters $m_1,m_2,m'_1,m'_2,m''_1$ and $m''_2$ in (\ref{roots}) form a basis of the space of the reflection representation, and thus $S(V)$ can be identified with the algebra of polynomials in $m_1,m_2,m'_1,m'_2,m''_1,m''_2$.

The Weyl group $W(E_6)$ acts on $S(V)$ and the subalgebra of invariant polynomials for $W(E_6)$ is generated by 6 algebraically independent polynomials, which can be chosen homogeneous of degrees, respectively, 2,5,6,8,9,12 (the fundamental degrees of $E_6$).

The order of $W(E_6)$ is $51840$ and we can define an averaging operator over this group
\begin{equation}
 \langle. \rangle=\frac{1}{51840} \sum_{s \in W(E_6)} s \,.
\end{equation}
Then a choice of six fundamental homogeneous invariant polynomials for $W(E_6)$ is given by
\begin{eqnarray}
&& p_2=\frac{3}{2} \langle m_1^2 \rangle\ ,\quad p_5=\frac{8}{3}\langle (m'_1)^2m'_2m''_1m''_2 \rangle\ ,\quad p_6=10\langle m_1m_2m'_1m'_2m''_1m''_2\rangle\,,\\
&&  p_8=\frac{5}{3} \langle (m_1)^2m_2(m'_1)^2m'_2m''_1m''_2 \rangle \ ,\quad p_9=\frac{40}{27}\langle (m_1m_2m'_1)^2m'_2m''_1m''_2 \rangle\,,\nonumber\\
&& p_{12}=\frac{20}{3}\langle( m_1m_2m'_1m'_2m''_1m''_2)^2\rangle\,. \nonumber  
\end{eqnarray}

We can now state the following theorem about  $Z_2(sl(3))^{spec}$:
\begin{thm}\label{thm-sym}
 The algebra $Z_2(sl(3))^{spec}$ is invariant under the action of the Weyl group of type $E_6$ given by \eqref{eq:s1}-\eqref{eq:s3}.
 The polynomials $a_2$, $a_5$, $a_6$, $a_8$, $a_9$ and $a_ {12}$ can be expressed in terms of the fundamental invariant polynomials as follows
 \begin{eqnarray}
  && a_2=p_2 -3 \ , \quad a_5 =-p_5 \ , \quad a_6=p_6+  \frac{p_2^3}{9}+\frac{2p_2^2}{3}-\frac{3p_2}{2}+1\,,\\
  && a_8=-p_8 +\frac{p_2^4}{54} +\frac{p_2p_6}{12}  +\frac{p_2^3}{18}+\frac{p_6}{2}+\frac{p_2^2}{6}-\frac{p_2}{4}+\frac{1}{8}\,,\\
  &&a_9 = -p_9-p_5\left(\frac{p_2^2}{27}+\frac{p_2}{3}- \frac{1}{4}\right)\,,\\
  &&a_{12}= -p_{12}+\frac{35p_6^2}{12} +\frac{p_2^6}{36}+\frac{17p_2^3p_6}{72} -\frac{p_2^2p_8 }{18} -\frac{7p_2p_5^2}{18}
  +\frac{p_2^5}{162}-\frac{p_2p_8}{3}+\frac{p_2^2p_6}{36}-\frac{p_5^2}{4}\\
  &&\hspace{1cm}-\frac{13p_2^4}{108}+\frac{13p_8}{2}-\frac{13p_2p_6}{24}-\frac{19p_2^3}{54}-3p_6-\frac{11p_2^2}{12}+\frac{11p_2}{8}-\frac{11}{16}\,.\nonumber
 \end{eqnarray}
\end{thm}
\proof By direct (but long) computations, one can show that the polynomials $a_i$ ($i=2,5,6,8,9,12$) are given in terms of the invariant polynomials as stated in the theorem.
The invariance of $Z_2(sl(3))^{spec}$ follows. \endproof

\begin{rem} 
There exists an additional automorphism $r$ of $Z_2(sl(3))^{spec}$ with $r:X \to -X$, $r:Y \to Y$ and 
\begin{eqnarray}
r&:& m_1 \mapsto -m_1\,, \quad m_2 \mapsto -m_2\,, \quad m'_1 \mapsto -m'_1\,, \quad m'_2 \mapsto -m'_2\,, \quad m''_1 \mapsto -m''_1\,, \quad m''_2 \mapsto -m''_2\,.
  \nonumber 
\end{eqnarray}
This transformation $r$ on the parameters and all the $s_i$ generate a group of order 103680 which is the complete symmetry of the root system of $E_6$: the group generated by the Weyl 
reflections and the central symmetry.
\end{rem}

\begin{rem} 
Recall that we have described a group of automorphisms $\textbf{Aut}_0$ of $Z_2(sl(3))$ (before specialisation) of order 12. 
Of course, all these automorphisms can be pushed to the specialisation $Z_2(sl(3))^{spec}$. Their action on the parameters is as follows: we have all six permutations of the three pairs $(m_1,m_2)$, $(m'_1,m'_2)$ 
and $(m''_1,m''_2)$, together with their compositions with the map $\tau$ transposing $m_1\leftrightarrow m_2$, $m'_1\leftrightarrow m'_2$ and $m''_1\leftrightarrow m''_2$. 
Among them, the following six automorphisms leave $X$ and $Y$ invariant:
\[\textbf{Aut}_0^+=\{Id,\ (1,2,3),\ (1,3,2),\ (1,2)\tau,\ (1,3)\tau,\ (2,3)\tau \}\ .\]
They are included in the symmetry described in Theorem \ref{thm-sym} and corresponds to the action of some elements of the Weyl group $W(E_6)$.

So, modulo the symmetry under $W(E_6)$ which leaves $X$ and $Y$ invariant, the automorphisms in $\textbf{Aut}_0$ provide one more non-trivial automorphism 
(sending $X$ to $-X$ and $Y$ to $Y$). Modulo $W(E_6)$, this additional symmetry is exactly the central symmetry $r$ of the preceding remark. 
For example, it is easy to check that the composition $r\circ\tau$ corresponds to the element of $W(E_6)$ which is the longest element of the subgroup of type $A_2\times A_2\times A_2$.
\end{rem}

\subsection{A link with symplectic reflection algebra}
Reproducing the same reasoning as in Remark \ref{rem-gen}, we find that the specialised centraliser $Z_2(sl(3))^{spec}$ is isomorphic to a certain quantum Hamiltonian reduction algebra:
\begin{equation}\label{qHred}
\Bigl(\bigl(U_{(m_1,m_2)}\otimes U_{(m'_1,m'_2)}\otimes U_{(m''_1,m''_2)}\bigr)/I\Bigr)^{sl_3}\cong Z_2(sl(3))^{spec}\ ,
\end{equation}
where each $U_{(m,n)}$ is the quotient of $U(sl(3))$ by the annihilator of the Verma module with highest weight $(m,n)$, and $I$ is the left ideal defined similarly as in Remark \ref{rem-gen}. Indeed, it is a standard fact that the quotient $U_{(m,n)}$ is obtained by fixing the values of the Casimir elements of $U(sl(3))$ according to the weight $(m,n)$. So the specialisation of the three pairs of Casimir elements in the left hand side of (\ref{qHred}) exactly corresponds to the specialisation in Definition \ref{defAspec}.

A particular case of the results of \cite{ELOR} asserts that the Hamiltonian reduction above is isomorphic to the spherical subalgebra $eH(\Gamma)e$ of a certain symplectic reflection $H(\Gamma)$ algebra of rank 1. The symplectic reflection algebra $H(\Gamma)$ of rank 1 corresponds to a finite subgroup $\Gamma$ of $SL_2(\mathbb{C})$ and is a certain quotient of the smashed product $T(\mathbb{C}^2)\rtimes \mathbb{C}[\Gamma]$. The idempotent $e$ is the idempotent $\frac{1}{|\Gamma|}\sum_{g\in\Gamma}g$ in $\mathbb{C}[\Gamma]$. In our situation, the group $\Gamma$ is the one associated through the McKay correspondence to the Dynkin diagram of type $E_6$. This is where the type $E_6$ appears. We note that the symplectic reflection algebra $H(\Gamma)$ contains $6$ parameters, one for each non-trivial conjugacy class in $\Gamma$.  The correspondence between these 6 parameters and the 3 $sl(3)$-weights is quite intricated and is described in \cite{ELOR}.

Thus the algebraic description of $Z_2(sl(3))^{spec}$ obtained in Proposition \ref{pr:Aspec} and  Theorem \ref{thm-sym} can be seen as an explicit presentation of the spherical subalgebra $eH(\Gamma)e$. It seems not obvious to obtain this presentation directly from the definition of $eH(\Gamma)e$.

\section{A connection between $\alg$ and the Racah or Hahn algebras \label{sec:RH}}

As mentioned in Section \ref{sec-A}, special cases of the algebra $\alg$ cover the known Racah and Hahn algebras. 
These special cases correspond to taking some of the central parameters to be 0. 
In this section, we offer a somewhat different connection between the algebra $\alg$ and the Racah or Hahn algebras. 
Namely we explain that the defining relations of $\alg$ are satisfied by some general operators of Heun type in the Racah algebra and also in the Hahn algebra. 
This suggests some connections of the algebra $\alg$ with the Heun--Racah \cite{BCTVZ} and the Heun--Hahn algebras \cite{VZ}.

Due to several free parameters in the Heun operators, many possible values of the central parameter of $\alg$ can be realised this way. 
Nevertheless we do not investigate here the possibility of realising $\alg$ with arbitrary central parameters in the Racah or Hahn algebras.

\subsection{Racah algebra. \label{sec:racah}}

The Racah algebra $\cR$ is the algebra generated by the following elements \cite{GVZ}:
\[ R_1\ ,R_2,\ R_3\ \ \text{and}\ \  d,\ e_1,\ e_2, \ \Gamma\,,\]
with the defining relations:
\begin{equation}\label{relRAcah}
\begin{array}{l}
d,\ e_1,\ e_2, \Gamma\ \text{are central\,,}\\[0.5em]
[R_1,R_2]=R_3\,,\\[0.5em]
[R_3,R_1]= R_1^2+\{R_1,R_2\}+ d R_1 + e_2\,,\\[0.5em]
[R_2,R_3]= R_2^2+\{R_1,R_2\}+ d R_2 +e_1\,,\\[0.5em]
\Gamma=\{R_1^2,R_2\}+\{R_1,R_2^2\} + R_1^2+R_2^2+R_3^2+(d+1)\{R_1,R_2\}+(2e_1+d)R_1+(2e_2+d)R_2 \  .
\end{array}
\end{equation}
The following polynomials in terms of the generators of the Racah algebra $\cR$:
\begin{equation}\label{eq-morR}
A= z_0+z_1 R_1 +z_2 R_2 +2 R_3\,,\qquad B= z_3+\frac{z_1(z_1+z_4)}{2} R_1 +\frac{z_2(z_2+z_4)}{2} R_2 +z_4 R_3 +2 \{R_1,R_2\} \,,
\end{equation}
satisfy relations \eqref{relAlg} of the algebra $\alg$ with the parameters $a_0=-6$ and $a'_0=-2$ and the other parameters $a_1, a_2, a_3, a_4, a_5,a_6, a_8$ and $a_9$
functions of $z_0,z_1,z_2,z_3,z_4,\delta,\epsilon_1,\epsilon_2$ and $\Gamma$. These expressions are obtained by a straightforward computation. We provide their explicit expressions in Appendix \ref{App:A}.

The generators $A$ and $B$ given by \eqref{eq-morR} are the Heun--Racah operators.
The so-called Heun--Racah algebra is realised by taking $A$ or $B$ as one generator and for the other, one of the generators of the Racah algebra $R_1$ or $R_2$. The above result 
means that the algebra $\alg$ can be obtained from two different Heun--Racah operators.

\subsection{Hahn algebra. \label{sec:hahn}}

The Hahn algebra $\cH$ \cite{Zh,VZ} is the algebra generated by the following elements:
\[ H_1\ ,H_2,\ H_3\ \ \text{and}\ \  \delta,\ \epsilon_1,\ \epsilon_2, \ \Lambda\,,\]
with the defining relations:
\begin{eqnarray}
&& \delta_1, \delta_2,\ \epsilon_1,\ \epsilon_2,\ \Lambda\ \text{are central\,,}\\[0.5em]
&&H_3=[H_1,H_2]\ , \\
&&[H_2,H_3]= -2 \{H_1,H_2\} +\delta_2 H_2 +\delta_1 H_1 +\epsilon_1 \ , \\
&&[H_1,H_3]= 2 H_1^2 -\delta_2 H_1 + H_2 +\epsilon_2 \ , \\
&&H_3^2 =\Lambda+2\{H_1^2,H_2\}-\delta_2 \{H_1,H_2\}+2\epsilon_2 H_2-(4+\delta_1)H_1^2+ H_2^2-2(\epsilon_1-\delta_2)H_1 \ .
\end{eqnarray}

Let us define the following Heun--Hahn operator by
\begin{equation}
 A= z_0+ z_1 H_1 +z_2 H_2 +z_3 H_3 +z_4 \{H_1,H_2\}\ . \label{eq:Xe}
\end{equation}
We define also the operator $B$ as follows
\begin{eqnarray}
 B=z_5+z_6 H_1+z_7 H_2+z_8 H_3+z_9 \{H_1,H_2\} +z_{10} H_1^2+z_{12} H_1^2 H_2+z_{13} H_1 H_2H_1 \,, \label{eq:Ye}
\end{eqnarray}
where 
\begin{eqnarray}
 z_1&=&\mp \frac{\delta_1}{4}\mp (z_{13}+1)(z_{13}-2z_{3}+1)\ , \quad z_4=\pm \frac{1}{2}\,,\\
 z_6&=& \frac{\epsilon_1}{2} +\left( \frac{\delta_2}{4}\pm z_2\right)\frac{\delta_1}{3}+\frac{\delta_2}{6}(z_{13}+1)(2z_{13} \pm 2z_3 -1)-\frac{2z_2}{3}(z_{13}+1)(z_{13}\mp2z_3+1)  \,,\\
z_7&=&\frac{1}{4}(2z_3\mp 1)(2z_3 \mp 2z_{13}\mp 1)  -\frac{1}{3} z_2^2  \pm \frac{\delta_2z_2}{6}\,,\quad z_8=\frac{z_2}{3}(2z_3\mp 3z_{13}\mp 3)   \pm \frac{\delta_2z_3}{6} \,,\\
z_9&=&\frac{\delta_2}{12}\mp \frac{2z_2}{3}\,, \quad z_{10}= \frac{\delta_1}{4}-z_{13}(z_{13}+1)\,,\quad z_{12}=-1-z_{13}\,.
\end{eqnarray}
The elements $A$ and $B$ satisfy relations \eqref{relAlg} of the algebra $\alg$ with the parameters $a_0=-6$, $a'_0=-2$, $a_1=0$ and the other parameters $ a_2, a_3, a_4, a_5,a_6, a_8$ and $a_9$
functions of $z_0,z_2,z_3,z_5,\delta_1,\delta_2,\epsilon_1,\epsilon_2$ and $\Lambda$. In Subsection \ref{sec:racah}, these functions are obtained by a straightforward but lengthy computation. 
In this case they are quite complicated and are not provided.

The element $A$ given by \eqref{eq:Xe} is the Heun--Hahn operator whereas the element $B$ is a generalisation (we recover the Heun--Hahn operator for $z_{10}=z_{12}=z_{13}=0$).
The so-called Heun--Hahn algebra is realised by $A$ and one of the generators of the Hahn algebra that is $H_1$ or $H_2$. 

\section{Conclusion}

This study has unraveled elegant algebraic structures associated to the Clebsch--Gordan series (or the Littlewood--Richardson coefficients) of $sl(3)$ which suggest various fascinating perspectives. 
Let us mention a few to conclude. It should be possible to throw more light on the solution of the corresponding ``missing label problem'' by exploiting the connection with quadratic algebras and their representations. The Bethe ansatz could be brought to bear on the diagonalisation of
the generators of $Z_2(sl(3))$, the labelling operators, viewed as Hamiltonians. This will be the subject of a forthcoming publication.

Besides, the results presented here provide the groundwork to 
identify the full algebra associated to the Clebsch--Gordan problem for $sl(3)$ which should have for generators the complete set of
operators labelling both the direct product states and the recoupled ones. The determination of this algebra would provide for $sl(3)$ the analogue
of the Hahn algebra attached to the Clebsch--Gordan problem for $sl(2)$.

The Calabi--Yau properties of the algebra $\mathcal{A}^{gen}$ and their applications deserve further scrutiny. Also, the symmetry of the specialised diagonal
centraliser of $sl(3)$, $Z_2(sl(3))$, under the action of the Weyl group of $E_6$  lifts the veil over striking facts that stand to significantly inform the
representation theory of potential algebras such as those of Racah and Askey--Wilson. It should be recalled that the particular case of $\mathcal{A}^{gen}$ whose quotient is $Z_2(sl(3))$, had also been found in connection with the labelling of representation basis vectors associated to the chain $su(3) \supset o(3) \supset o(2)$. This is awaiting an analysis similar to the one carried here. Many other generalisations also come to mind. On the one hand there are
additional well known missing label problems such as the ones corresponding to the subalgebra chains: $su(4) \supset su(2)\oplus su(2)$, $o(5) \supset su(2) \oplus u(1), o(5) \supset o(3), g_2 \supset o(3)$ etc. What is the structure of the corresponding centraliser? What symmetries will operate?
On the other hand, the same questions can be asked when considering diagonal embeddings in multifold tensor products. We plan on examining these
issues in the future.

\appendix

\section{Parameters of $\alg$ for the realisation in terms of Heun--Racah operators \label{App:A}}

The generators $A$ and $B$ given by \eqref{eq-morR} satisfy relations \eqref{relAlg} with the parameters given by
\[\begin{array}{ll}
a_0= & -6\,,\\[0.5em]
a_4= & 4d^2+2(z_1z_2-2)d-16(e_1-e_2)-4z_0(z_1+z_2)-6z_0z_4+12z_3\,,\\[0.5em]
a_2= & -\frac{1}{2}(z_1^2+z_2^2+3z_4^2+3z_1z_2)-2z_4(z_1+z_2)+2d-2  \,,\\[0.5em]
a_5= &-2(2z_1+2z_2+3z_4)z_3-4dz_0+(z_1^2+3z_1z_2+4z_1z_4+z_2^2+4z_2z_4+3z_4^2+4)z_0\\
 & +(2z_4-\displaystyle \frac{1}{2}z_1z_2(z_1+z_2)-z_1z_2z_4)d+8(z_2+z_4)e_1+8(z_1+z_4)e_2-2d^2z_4\,,\\[0.5em]
c_1= &2(2z_1+2z_2+3z_4)z_0z_3-8(z_2+z_4)e_1z_0-8(z_1+z_4)e_2z_0+16(e_1+e_2)z_3-32e_1e_2-6z_3^2 \\
& +(2d-2-\frac{1}{2}(z_1^2 +z_2^2 +3z_1z_2+3z_4^2  )-2(z_1+z_2)z_4)z_0^2+\frac{1}{2}dz_0z_1z_2(z_1+z_2) \\
& +d(z_0z_4-2z_3)(z_1z_2+2d-2)+\frac{1}{2}(z_2^2-4)(z_1z_2-z_1^2+4d)e_1+\frac{1}{2}(z_1^2-4)(z_1z_2-z_2^2+4d)e_2\\
& +2\Gamma(z_1^2-z_1z_2+z_2^2-4d)\ ,\\[0.5em]
 a_1 = & 2z_1+2z_2+3z_4\ , \\[0.5em]
a'_0= & -2\,,\\[0.5em]
a_3= &-\frac{1}{4}(z_1+z_2)(3z_1z_2+8)-\frac{3}{4}z_4^3+3dz_4+6z_0-3z_4-\frac{3}{4}(z_1^2-3z_1z_2-z_2^2)z_4-\frac{3}{2}(z_1+z_2)z_4^2  \,,\\[0.5em]
a_6= &(z_4^2-\frac{1}{2}z_1z_2(z_4^2+z_1z_2+ z_2  z_4+  z_1   z_4)-2  z_1  z_2) d-( z_1^2+3  z_1  z_2+ z_2^2+4  (z_1+z_2)  z_4+3  z_4^2+4)  z_3\\
&+2(  z_2^2+4  z_2  z_4+2  z_4^2+4) e_1+\frac{3}{2} ( (z_1^2 + z_2^2+3  z_1  z_2+ z_4^2 )  z_4 +  z_1^2  z_2+  z_1  z_2^2 )  z_0\\
&+2(  z_1^2+4  z_1  z_4+2  z_4^2+4) e_2-6  z_0^2-d^2  z_4^2-6 d  z_0  z_4+4 d  z_3\\
&+( (3  z_4^2+4)(  z_1+ z_2)+6  z_4)  z_0+8 \Gamma \,,\\[0.5em]
a_9= & d^2  z_0  z_4^2-4 d  z_0  z_3-\frac{1}{4}  z_1 ( z_2^2-4)( z_2+ z_4) ( z_1- z_2) e_1+\frac{1}{4}  z_2 ( z_1^2-4)  ( z_1+ z_4) ( z_1- z_2) e_2-2 d^2  z_3  z_4\\
&+ z_4 ( z_2^2-4)  d e_1+ z_4 ( z_1^2-4) d e_2+8( z_2+ z_4) e_1  z_3+8(  z_1+  z_4) e_2  z_3-16 e_1 e_2  z_4\\
& -(\frac{3}{4}(z_4^3+3  z_1  z_2  z_4 +  z_1^2  z_2+  z_1^2  z_4+  z_1  z_2^2+ z_2^2  z_4 )+3  z_4+\frac{3}{2}  z_4^2 ( z_1+ z_2)+2(  z_1+ z_2) )  z_0^2\\
&-(2  z_1+2  z_2+3  z_4)  z_3^2+2  z_0^3+3 d  z_0^2  z_4+( \frac{1}{2} z_1  z_2 ( z_1  z_4  +   z_2  z_4 +   z_4^2+ z_1  z_2 )+2  z_1  z_2- z_4^2) d  z_0\\
&-2(  z_2^2+4 z_2  z_4+2  z_4^2+4) e_1  z_0-2( z_1^2+4  z_1  z_4+2 z_4^2+4) e_2  z_0-z_1  z_2(  z_4+\frac{1}{2}(z_1+z_2)) d  z_3\\
&+( z_1^2+3  z_1  z_2+4  z_1  z_4+ z_2^2+4  z_2  z_4+3  z_4^2+4)  z_0  z_3+2dz_3z_4\\
&+\Gamma(z_1z_2(z_1+z_2)  -8z_0+(z_1^2-z_1z_2+z_2^2-4d)z_4)\ .
\end{array}\]

\vskip .5cm
\textbf{Acknowledgements.} The authors warmly thank Rupert Yu and Alexei Zhedanov for stimulating discussions and are very grateful to the anonymous referee for valuable comments. N.Cramp\'e and L.Poulain d'Andecy are partially supported by Agence National de la Recherche Projet AHA ANR-18-CE40-0001.
 L.Poulain d'Andecy is grateful to the Centre de Recherches Math\'ematiques (CRM) for 
 hospitality and support during his visit to Montreal in the course of this investigation.
 The research of L.Vinet is supported in part by a Discovery Grant from the Natural Science and Engineering Research Council (NSERC) of Canada.

\end{document}